\documentclass[namedate,webpdf]{amsart}

\newtheorem{theorem}{Theorem}%
\newtheorem{proposition}[theorem]{Proposition}%

\newtheorem{remark}{Remark}%
\newtheorem{fact}{Fact}%
\newtheorem{lemma}[theorem]{Lemma}%
\newtheorem{corollary}[theorem]{Corollary}%
\newtheorem{definition}{Definition}

\numberwithin{equation}{section}
\usepackage{booktabs}
\usepackage{natbib} 
\usepackage{hyperref}
\usepackage{xcolor}
\usepackage{amsaddr}
\usepackage{graphicx} 
\usepackage{float}
\restylefloat{table}

\hypersetup{
    colorlinks,
    linkcolor={red!},
    citecolor={blue!},
    urlcolor={blue!}
}

\begin{document}

\bibliographystyle{abbrvnat}
\setcitestyle{authoryear,semicolon,open={(},close={)}} %

\title[Underdetermined Fourier Extensions for Surface PDEs]{Underdetermined Fourier Extensions for Surface Partial Differential Equations}

\author{Daniel R. Venn and Steven J. Ruuth}
\address{Mathematics, Simon Fraser University, 8888 University Drive, Burnaby, V5A 1S6, BC, Canada}
\email{dvenn@sfu.ca}

\begin{abstract}We analyze and test using Fourier extensions that minimize a Hilbert space norm for the purpose of solving partial differential equations (PDEs) on surfaces. In particular, we prove that the approach is arbitrarily high-order and also show a general result relating boundedness, solvability, and convergence that can be used to find eigenvalues. The method works by extending a solution to a surface PDE into a box-shaped domain so that the differential operators of the extended function agree with the surface differential operators, as in the Closest Point Method. This differs from approaches that require a basis for the surface of interest, which may not be available. Numerical experiments are also provided, demonstrating super-algebraic convergence. Current high-order methods for surface PDEs are often limited to a small class of surfaces or use radial basis functions (RBFs). Our approach offers certain advantages related to conditioning, generality, and ease of implementation. The method is meshfree and works on arbitrary surfaces (closed or non-closed) defined by point clouds with minimal conditions.
\end{abstract}

\global\long\def\b#1{\left(#1\right)}%

\global\long\def\p{\phi}%

\global\long\def\abs#1{\left|#1\right|}%

\global\long\def\oiin{\oiint}%

\global\long\def\d{\text{d}}%

\global\long\def\curl{\nabla\times}%

\global\long\def\div{\nabla\cdot}%

\global\long\def\pd#1#2{\frac{\partial#1}{\partial#2}}%

\global\long\def\eval{\biggr|}%

\global\long\def\sb#1{\left[#1\right]}%

\global\long\def\cb#1{\left\{  #1\right\}  }%

\global\long\def\lc{\varepsilon}%

\global\long\def\R{\mathbb{R}}%

\global\long\def\C{\mathbb{C}}%

\global\long\def\code#1{\mathtt{#1}}%

\global\long\def\N{\mathbb{N}}%

\global\long\def\K{\mathbb{K}}%

\global\long\def\O{\mathcal{O}}%

\global\long\def\lb#1{\left(#1\right.}%

\global\long\def\rb#1{\left.#1\right)}%

\global\long\def\ss{\subseteq}%

\global\long\def\norm#1{\left\Vert #1\right\Vert }%

\global\long\def\ip#1{\left\langle #1\right\rangle }%

\global\long\def\B{\mathcal{B}}%

\global\long\def\ra{\mathcal{R}}%

\global\long\def\D{\mathcal{D}}%

\global\long\def\vp{\varphi}%

\global\long\def\sff#1#2{\left\langle #1,#2\right\rangle _{\text{II}}}%

\global\long\def\W{\mathcal{W}}%

\global\long\def\F{\mathcal{F}}%

\global\long\def\C{\mathbb{C}}%

\global\long\def\cont{\supseteq}%

\global\long\def\cp{\text{cp}}%

\global\long\def\vc#1{\boldsymbol{#1}}%

\global\long\def\x{\vc x}%

\global\long\def\y{\vc y}%

\global\long\def\n{\hat{\vc n}}%

\global\long\def\argmin#1{\underset{#1}{\text{argmin}}}%

\global\long\def\L{\mathcal{L}}%

\global\long\def\H{\mathcal{H}}%

\maketitle

\section{Introduction}

Partial differential equations (PDEs) on surfaces appear in a variety of contexts, often in medical imaging or computer graphics. 
Various methods for solving such problems have been developed over the past couple of decades; however, the number of methods with proofs of convergence is fairly limited. 

Among existing approaches, surface finite element methods are likely the best understood, owing largely to the significant work of Dziuk and Elliott in the 2010s \citep[see][]{dziuk13}. A number of implementations of radial basis function (RBF) methods, often using least squares collocation, have also been studied. RBFs are well understood for interpolation and can be high order, and while the popular method of \cite{kansa90} is known to potentially fail, techniques based on oversampling have been successful on flat domains as well as surfaces when a point cloud regularity condition is satisfied \citep{chen20}. The Closest Point Method (CPM) \citep{ruuth08} has been extensively studied numerically, and consistency follows directly from consistency of the underlying interpolation and finite difference techniques, but stability is not yet fully understood, though the method works in practice in most cases. RBF methods using similar extension approaches as the CPM, either by performing a Closest Point extension directly \citep{petra18}, or by satisfying normal derivative conditions \citep{piret12}, have also been proposed.

We propose a new, flexible, and high-order class of meshfree methods for PDEs on surfaces using underdetermined Fourier extensions. These extensions are chosen to be norm-minimizing in a certain Hilbert space. We show that the norm-minimizing property is enough to prove statements regarding the methods' convergence for a variety of problems. Finally, we connect our method to certain implementations of symmetric Hermite-Birkhoff interpolation with RBFs and demonstrate how it could be used to investigate the solvability of PDEs.

Norm-minimizing Hermite-Birkhoff interpolants have been previously studied in the RBF literature through Hermite RBFs \citep[see, for example,][]{frank98, sun94}. More recently, Chandrasekaran, Gorman, and Mhaskar have studied solving the norm-minimization problem more directly \citep{chand18}, which has some numerical advantages. Namely, solving the optimization problem directly can improve conditioning and is often simpler to implement. In our work, we explore the direct solution of the norm-minimization problem for surface PDEs via function extension from surfaces to a box. Extension does not require a basis for functions on the surface; a standard Fourier basis can be used.

Along with numerical tests in Section \ref{sec:numexp}, we present new analytical results in Propositions \ref{prop:gamma}, \ref{prop:Let--and}, and \ref{prop:conv} showing the high-order nature of our extension approach. We also prove a very general result regarding the relationship between boundedness, PDE solvability, and convergence in Proposition \ref{prop:Let--be2}. In particular, we show that boundedness of numerical solutions to a PDE as point spacing goes to zero implies solvability of the non-discretized PDE and convergence of the numerical solution to a PDE solution, regardless of whether the problem has multiple solutions. This result is used to find eigenvalues in Subsection \ref{subsec-eigs}.

\section{Preliminaries}
\label{sec:prelim}

\subsection{Fourier Extensions}

Spectral methods for PDEs typically work by searching for a PDE solution of the form:
\[
\tilde{u}=\sum_{n=1}^{N_{b}}a_{n}\phi_{n},
\]
where $\cb{a_{n}}_{n=1}^{N_{b}}$ are coefficients, and $\cb{\phi_{n}}_{n=1}^{N_{b}}$
are functions. $\cb{\phi_{n}}_{n=1}^{N_{b}}$ are often chosen to
be a subset of basis functions for some function space. Fourier basis functions are a common choice: $\phi_{n}\b{\x}=e^{i\vc{\omega}_{n}\cdot\x}$ for a set of frequencies $\cb{\vc\omega_n}$. However, for all but the simplest domains, coming up with suitable functions $\cb{\phi_n}$ can be challenging or impossible without an already existing, high-order numerical method.

A possible resolution to this problem is to avoid finding eigenfunctions
on a complicated domain altogether. Instead, the possibly complicated
domain $S$ is placed inside a box-shaped domain $\Omega\ss\R^{m}$. We can then use the box's Fourier basis functions
to expand functions on $S$, albeit not uniquely. In approximation
theory terms, our functions $\cb{\phi_{n}}_{n=1}^{N_{b}}$ will not
form a basis for functions on $S$, but rather a frame \citep[see][]{adcoc19}. 

Consider a reasonably nice function (where we will be more precise
later) $f:S\to\R$. There exist many functions $\tilde{f}:\Omega\to\R$
such that
\[
f=\tilde{f}\eval_{S}.
\]
The idea of Fourier extension is to attempt to write a Fourier series for $\tilde f$:
\[
\tilde{f}=\sum_{n=1}^{\infty}a_{n}\phi_{n},
\]
where $\cb {a_n}$ are coefficients and $\phi_n \b {\vc x} = e^{i\vc \omega_n \cdot \vc x}$ are Fourier basis functions for the box $\Omega$.

This sort of problem, where we extend $f$ from $S$ to a larger domain
$\Omega\supset S$ using a Fourier series, is referred to by \cite{boyd03} as Fourier extension of
the third kind when $S$ and $\Omega$ have the
same dimensionality. One motivation for computing such an extension
is that we are able to compute derivatives of Fourier series on $\Omega$
with high accuracy.

Computing a Fourier extension without knowledge of the function's exact Fourier coefficients is
typically done by only approximately satisfying $f=\tilde{f}$ on
$S$. A sample of points in $S$ is chosen, then the coefficients
$a_{n}$ are computed via least squares.
Certain implementations have been shown to converge super-algebraically in 1D for smooth
functions when using Chebyshev nodes, or oversampled, uniformly spaced nodes \citep{adcoc14}.

Fourier extensions have been applied successfully in various contexts,
ranging from straightforward 1D function approximation, to surface
reconstruction \citep{bruno07}, to PDEs in flat domains
\citep{boyd05,stein17,bruno22,matth18}. Analytical results, however, remain
fairly sparse in higher dimensions, particularly for PDEs.

\subsection{Closest Point and Closest Point-like Extensions}

Once we start working on manifolds embedded in a higher dimensional
space, we have another motivation for performing extensions from a
manifold $S$ to a larger domain $\Omega$ of co-dimension zero. There
are many available methods for computing derivatives and solving PDEs
in flat, Euclidean spaces; conversely, solving PDEs on manifolds directly
can be quite challenging, especially if a connected mesh is not available.

A straightforward method for extending functions off of manifolds
is a Closest Point extension, which is the idea underlying the Closest
Point Method (CPM) for PDEs on surfaces \citep{ruuth08}.
\begin{definition}
The (Euclidean) Closest Point extension of a function $f:S\to\R$,
is the function $f\circ\cp_{S}$, where
\[
\cp_{S}\b{\x}=\argmin{\y\in S}\norm{\x-\y}_{2}.
\]
\end{definition}

If $S$ is a twice differentiable manifold, then $\cp_{S}$ is well-defined and differentiable on an open set in $\R^m$ containing $S$ \cite[see, for example,][Problem 6-5]{lee13}; $\cp_S$ in general is one order less smooth than the manifold.
The utility of such an extension is clear from the following facts
(see the work of \cite{ruuth08} for the original statements and \cite{marz12}
for detailed proofs).

\begin{fact}
\label{fact:If-,-then:}If $f\in C^{1}\b S$ where $S$ is a twice continuously differentiable
manifold, then 
\[
\nabla\b{f\circ\cp_{S}}\eval_{S}=\nabla_{S}f,
\]
where $\nabla_{S}$ is the gradient intrinsic to $S$.
\end{fact}

\begin{fact}
\label{fact:If-,-then:-1}If $f\in C^{2}\b S$ where $S$ is a thrice differentiable
manifold, then

\[
\Delta\b{f\circ\cp_{S}}\eval_{S}=\Delta_{S}f,
\]
where $\Delta_{S}$ is the Laplace-Beltrami operator on $S$.
\end{fact}

Combining these facts allows for various PDEs to be solved on surfaces
without a parametrization or a mesh.

As it turns out, Closest Point extensions are not the only extensions
where Facts \ref{fact:If-,-then:} and \ref{fact:If-,-then:-1} hold.
\cite{marz12} studied extensions of the form $f\circ P$, where
$P:\Omega\to S$ is idempotent, and found general conditions needed
for $P$ so that Facts \ref{fact:If-,-then:} and \ref{fact:If-,-then:-1}
still hold with $P$ instead of $\text{cp}_{S}$. 

We state a couple of facts needed for functions extended from $S$ to $\Omega$ in
general to be used for computing differential operators on a hypersurface $S$. First,
we need a definition.
\begin{definition}
\label{def:clopolike-1}Let $f\in C^{1}\b S$ where $S$ is a twice differentiable hypersurface. Then $\tilde{f}\in C^{1}\b{\Omega}$
is a first-order Closest Point-like extension of $f$ from $S$ when

\[
\n_{S}\cdot\nabla\tilde{f}\eval_{S}  =0\quad\text{and}\quad
\tilde{f}\eval_{S}  =f,
\]
where $\n_S$  is the normal vector to $S$.

\end{definition}

This definition is motivated by the next fact.
\begin{proposition}
\label{prop:Let--satisfy:}Let $f\in C^{1}\b S$ and $\tilde{f}\in C^{1}\b{\Omega}$
be a first-order Closest Point-like extension of $f$. Then
\[
\nabla\tilde{f}\eval_{S}=\nabla_{S}f.
\]
\end{proposition}

To work with the Laplace-Beltrami operator, we need an important
result adapted from \cite{xu03}.
\begin{lemma}\cite[Lemma 1]{xu03}
\label{lem:(Lemma-1-of} Let $f\in C^{2}\b S,\tilde{f}\in C^{2}\b{\Omega}$ where $S$ is a thrice differentiable hypersurface.
If
\[
\tilde{f}\eval_{S}=f,
\]
then
\[
\Delta_{S}f=\b{\Delta\tilde{f}-\kappa\n_{S}\cdot\nabla\tilde{f}-\n_{S}\cdot\b{D^{2}\tilde{f}}\n_{S}}\eval_S,
\]
where $\kappa=\nabla_{S}\cdot\n_{S}$ is the mean curvature (sum of
principal curvatures in our convention) of $S$.
\end{lemma}

Using Lemma \ref{lem:(Lemma-1-of}, we can then
compute $\Delta_{S}$ using a first-order Closest Point-like extension.
\begin{corollary}
\label{cor:Let--and}Let $f\in C^{2}\b S$ and let $\tilde{f}\in C^{2}\b{\Omega}$
be a first-order Closest Point-like extension of $f$. Then
\[
\nabla\tilde{f}\eval_{S} =\nabla_{S}f,\quad
\b{\Delta\tilde{f}-\n_{S}\cdot\b{D^{2}\tilde{f}}\n_{S}}\eval_{S} 
=\Delta_{S}f.
\]
\end{corollary}

Crucially, the only hypersurface information we need is $\n_{S}$, not any of its
derivatives or curvature information. Also note that Corollary \ref{cor:Let--and} only imposes one additional interpolation condition. Some previous methods, such as the higher-order version of Piret's orthogonal gradients method with RBFs \citep{piret12}, impose two additional conditions; $\n_s \cdot \nabla_S f$ and $\n_{S}\cdot\b{D^{2}\tilde{f}}\n_{S}$ are both set to zero. Using only one condition is more computationally efficient. We use this corollary to solve surface PDEs in Section \ref{sec:numexp}.

We now note that smooth, periodic, Closest Point-like extensions exist for suitable hypersurfaces and functions.

\begin{proposition}\label{prop-cpexist}
Let $S\subset \Omega$ be a $C^{p+1}$ hypersurface such that $\overline{S}\subset T$, where $T\subset\Omega$ is also a $C^{p+1}$ hypersurface.
For $p\ge1$,  $f\in C^{p}\b T$, there
exists a periodic function with $p$ periodic derivatives $\tilde{f}\in C^{p}\b{\Omega}$ such that $\tilde{f}$
is a first-order Closest Point-like extension of $f$ from $S$.
\end{proposition}

This is constructed as $\tilde f = \b{f\circ\cp_T}$ in an open set $U$ containing $T$, then extended from a closed subset of $U$ with an interior that contains $\overline S$ to $\Omega$ using a partition of unity as in Lemma 2.26 of the text by \cite{lee13}, such that the support of $\tilde f$ is compactly contained in $\Omega$. 
This result tells us that for the PDEs on surfaces we consider, there are periodic extended solutions with the same degree of smoothness as the solution on the surface. In particular, the constrained norm-minimization problems that we set up will have feasible solutions.

\subsection{Radial Basis Function Interpolation\label{subsec:Interpolation}}

We will make use of existing interpolation results to prove the convergence
of our own methods; we therefore review some basic results regarding RBF interpolation.

Radial basis functions (RBFs) are a class of functions used for interpolation
and numerical PDEs that depend on the location of the interpolation
points $\cb{\x_{k}}_{k=1}^{\tilde{N}}\subset S\subset\R^{n}$. This
is vital since, due to the Mairhuber-Curtis Theorem \citep{curti59,mairh56},
any set of $\tilde{N}$ functions that does not depend on the location
of interpolation points cannot uniquely interpolate every set of $\tilde{N}$
points. Specifically, polynomial interpolation when the number of
polynomials matches the number of points can fail to produce a solution.

Before stating some results, we first define the fill distance.

\begin{definition}
Let $\cb{\x_{k}}_{k=1}^{\tilde{N}}\subset S\subset\R^{m}$ be distinct
points. Define the fill distance of $\cb{\x_{k}}_{k=1}^{\tilde{N}}$
in $S$:
\[
h_{\text{max}}:=\sup_{\x\in S}\min_{k\in\cb{1,2,\ldots,\tilde{N}}}\norm{\x-\x_{k}}_{2}.
\]

In 1D, with $S=\sb{a,b}$ and $\cb{x_k}_1^{\tilde N}$ increasing, this is equivalent to
\[
h_{\text{max}}=\max\b{\cb{\frac{1}{2}\abs{x_{k}-x_{k-1}}}_{k=2}^{\tilde{N}}\cup\cb{x_{1}-a,b-x_{\tilde{N}}}}.
\]
\end{definition}

That is, the fill distance is the largest distance between a point
$\x\in S$ and its closest point in $\cb{\x_{k}}_{k=1}^{\tilde{N}}$.

As suggested by the name, RBFs are radially symmetric. We also restrict
ourselves to the discussion of positive definite RBFs, which we now define.
\begin{definition}
A positive definite radial basis function is a function $\phi:\R^{m}\to\R$
that possesses two properties:
\begin{enumerate}
    \item $\phi$ is radial;
    \[
    \norm{\x}_{2}=\norm{\y }_{2}\implies\phi\b{\x}=\phi\b{\y}.
    \]
    \item $\phi$ is (strictly) positive definite. That is, if we define the interpolation matrix $\vc{\Phi}$ on any set of unique points $\cb{\x_{k}}_{k=1}^{\tilde{N}}\subset \R^m$:
    
    \[
    \vc{\Phi}_{jk}=\phi\b{\x_{j}-\x_{k}},
    \]
    then $\vc\Phi$ is a symmetric positive definite matrix.
\end{enumerate}
\end{definition}

Note that this means linear combinations of the functions $\cb{\psi_{k}}_{k=1}^{\tilde{N}}$ defined by $\psi_{k}\b{\x}=\phi\b{\x-\x_{k}}$ can interpolate any function
on $\cb{\x_{k}}_{k=1}^{\tilde{N}}$ (since $\vc\Phi$ must be positive definite). Given a function $f$ and RBF $\phi$, we refer to the unique function $\tilde f \in \text{span}\cb{\psi_k}$ such that $\tilde{f}\b {\x_k}=f\b{\x_k}$ for each $k\in\cb{1,2,\ldots,\tilde N}$ as the RBF interpolant of $f$ on $\cb{\x_k}_{k=1}^{\tilde N}$ using $\phi$. Our goal now is to quickly prove a general statement regarding functions with scattered zeros. 

First, we need a definition \citep[see][Def. 3.6]{wendl04}
to restrict the types of domains we can consider.
\begin{definition}
A set $U\subset\R^{m}$ satisfies an interior cone condition
if there exists an angle $\theta\in\b{0,\frac{\pi}{2}}$ and a radius
$r>0$ such that for all $\x\in U$, there exists some unit vector
$\vc{\xi}\b{\x}$ so that
\[
\cb{\x+\lambda\vc y:\y\in\R^{m},\norm{\y}_{2}=1,\y\cdot\vc{\xi}\b{\x}\ge\cos\theta,\lambda\in\sb{0,r}}\ss U.
\]
\end{definition}

In other words, $U$ satisfies an interior cone condition if there is
some cone of a fixed size and angle that can be placed with its vertex
at each $\x\in U$ and remain entirely contained in $U$.
Informally, $U$ is only “finitely pointy”. Hypersurfaces do not satisfy an interior cone condition; the hypersurfaces we consider will require that the domain $U$ of a local chart satisfies an interior cone condition instead.

The next proposition is adapted from \cite{wendl04}.
\begin{proposition}\cite[Thms. 10.35 and 11.17]{wendl04}
Let $f\in H^{\frac{m}{2}+q+\frac{1}{2}}\b{U}$ with $m\ge3$
if $q=0$ and $m\in\N$ for $q>0$ where %
$U \subset \R^m$ is bounded and
satisfies an interior cone condition. Then there exists a radial basis
function $\phi_{m,q}$ such that if $\tilde{f}$ is the RBF interpolant of $f$ on $\cb{\x_k}_{k=1}^{\tilde N}$ using $\phi_{m,q}$ where $\cb{\x_k}_{k=1}^{\tilde N}$ has fill distance
$h_{\text{max}}$ on $U$, then there exist constants $C_{m,q,\abs{\alpha}},h_{0,m,q,\abs{\alpha}}>0$
such that as long as $h_{\text{max}}\le h_{0,m,q,\abs{\alpha}}$,

\[
\norm{\partial^\alpha f- \partial^\alpha\tilde{f}}_{L^{\infty}\b{U}}\le C_{m,q,\abs{\alpha}}h_{\text{max}}^{q+\frac{1}{2}-\abs{\alpha}}\norm f_{H^{\frac{m}{2}+q+\frac{1}{2}}\b{U}},
\]
where $\alpha$ is a multi-index with $\abs{\alpha}\in\cb{0,1,2,\ldots,q}$.
\end{proposition}

Noting that the RBF interpolant of a function that vanishes on $\cb{\x_{k}}_{k=1}^{\tilde{N}}$
is simply the zero function, we then have a general result for functions vanishing on scattered points.
\begin{corollary}
\label{cor:Let--with}Let $f\in H^{\frac{m}{2}+q+\frac{1}{2}}\b{U}$
with $m\ge3$
if $q=0$ and $m\in\N$ for $q>0$ where %
$U$ is bounded and satisfies an interior cone condition. Assume $f$ vanishes
on $\cb{\x_{k}}_{k=1}^{\tilde{N}}$ with fill distance $h_{\text{max}}$
on $U$, then for all $p\in\cb{0,1,\ldots q}$ for $m\ge3$ and
$p\in\cb{1,2,\ldots q}$ for $m<3$, there exist constants $C_{m,p,\abs{\alpha}},h_{0,m,p,\abs{\alpha}}>0$
such that as long as $h_{\text{max}}\le h_{0,m,p,\abs{\alpha}}$,
\[
\norm{ \partial^\alpha f}_{L^{\infty}\b{U}}\le C_{m,p,\abs{\alpha}}h_{\text{max}}^{p+\frac{1}{2}-\abs{\alpha}}\norm f_{H^{\frac{m}{2}+p+\frac{1}{2}}\b{U}}.
\]
where $\alpha$ is a multi-index with $\abs{\alpha}\in\cb{0,1,2,\ldots,p}$.
\end{corollary}

Similar results regarding error bounds for functions with scattered zeros exist in the
literature and could also be used for efficient convergence proofs for a wide variety of methods, 
including our own. In particular, more general statements for convergence in Sobolev norms have been shown by \cite{narco05}.

\subsubsection{Application to Other Interpolation Methods}

Corollary \ref{cor:Let--with} has
implications for other interpolation methods, noting that it is
a general result; it does not have any specific relation to RBF interpolation other than the method of proof. This gives an important insight into constructing effective interpolation
methods. If we are able to bound various $H^k$ norms of our error $u - \tilde{u}$, where $\tilde{u}$ is an interpolant and $u$ is the original function, then Corollary \ref{cor:Let--with} guarantees uniform, high-order convergence of $\tilde{u}$ to $u$ as the fill distance goes to zero.

We also note that a version of Corollary \ref{cor:Let--with} holds on manifolds $S$ as well, if the manifold
is sufficiently smooth. While the manifold itself may not satisfy
an interior cone condition, the domain $U$ of a local parametrization
$\sigma:U\to S$ typically will. The corollary can then be applied
to $f\circ\sigma$.

\section{Proposed Methods \& Analysis}

In order to solve PDEs on scattered data (such as a point cloud on
a surface), we need to be able to interpolate functions over scattered
data on arbitrary domains. RBFs, as discussed in Subsection \ref{subsec:Interpolation},
provide one way of doing this. 

Among existing approaches are interpolation methods
based on using a standard RBF basis, such as the method of \cite{piret12}. Like their flat domain
counterparts such as Kansa's original method \citep[see][]{kansa90}, these methods are often quite successful in practice but face some theoretical issues due to the possible singularity of
the interpolation matrix once derivatives are involved. As with most RBF methods, the accuracy of these methods is often severely limited by conditioning, leading to errors much higher than machine error.

Other existing approaches are RBF-FD methods, which are particularly useful for large, time-dependent problems
due to the sparsity of the differentiation matrices produced. These methods can either work combined with the Closest
Point Method \citep{petra18} or by more directly approximating the surface operators using estimates of a local level set \citep{alvar21}. Being local methods, such methods typically have limited orders of convergence. Furthermore, stability (for time-stepping) and non-singularity (for elliptic problems) of such methods are currently unknown. 

Lastly, various implementations using standard RBF basis functions and least-squares collocation have also been used. Significant analysis of these sorts of approaches, along with numerical testing, has been done by \cite{chen20}. A quasi-uniform point cloud is required, and a possibly quite dense point cloud of collocation points is required relative to density of the point cloud of RBF centres for guaranteed convergence. Alternatively, one may approximate a weak form of the PDE to achieve symmetry of the discretized Laplace-Beltrami operator, as in the work by \cite{yan23}, which overcomes a weakness of standard interpolation approaches, but requires knowledge of the distribution from which collocation points are sampled, which may not be available.

We seek to explore other interpolation methods based on
underdetermined Fourier extensions, which will allow for super-algebraic
convergence with minimal conditions on the point clouds that can be used. However, as we will see in Subsection \ref{subsec:adjrange}, 
our methods can also be viewed as a sort of periodic Hermite-Birkhoff
RBF interpolation with some numerical advantages.

Existing Fourier extension methods typically use an overdetermined system.
Such methods for function approximation and bulk PDE problems in 2D have been recently proposed
\citep[see, for example,][]{stein17, matth18, bruno22}, though convergence analysis is typically limited to 1D. Analysis
of overdetermined methods also relies on relating the discrete least squares problem to an $L^2$-norm
minimization problem through quadrature. Such an approach would not be feasible on all but a few surfaces
where high-order quadrature schemes can be developed.

\subsection{Hermite-Birkhoff Fourier Extension Problem}

We first set up a constrained optimization problem. Let $S\subset\R^{m}$
be our domain of interest and let $\Omega\supset S$ be a box. Select
a point cloud $S_{\tilde{N}}:=\cb{\x_{k}}_{k=1}^{\tilde{N}}\subset S$
and linear differential operators $\cb{\F_{k}}_{k=1}^{\tilde{N}}$
of maximum order $p\in\N$, where $\F_k = \sum_{\abs \alpha \le p} c_{k,\alpha}\partial^\alpha$ and the coefficient functions $c_{k,\alpha}$ are bounded and defined in a neighbourhood of $\x_k$ and $\partial^\alpha$ are the usual partial derivatives in $\R^m$. We note here that surface differential operators can be written in this form using extensions to the box $\Omega$ (see Lemma \ref{lem:(Lemma-1-of} and Corollary \ref{cor:Let--and}). Let $\cb{\phi_{n}}_{n=1}^{\infty}$ be
the Fourier basis for the box consisting of complex exponential eigenfunctions
of the Lapacian on $\Omega$ so that $\phi_{n}\b{\x}=e^{i\vc{\omega}_{n}\cdot\x}$
for some frequency $\vc{\omega}_{n}$. Let $d=\cb{d_n}_{n=1}^\infty$ be a sequence where each $d_n>0$, and
let $\cb{f_{k}}_{k=1}^{\tilde{N}}\subset\C$ be interpolation values. We then
consider the problem: 
\begin{align}
\text{minimize, over \ensuremath{b\in\ell^{2}}: } & \norm b_{\ell^{2}}\label{eq:hbproblem}\\
\text{subject to: } & \sum_{n=1}^{\infty}d_{n}^{-\frac{1}{2}}b_{n}\b{\F_{k}\phi_{n}}\b{\x_{k}}=f_{k}\text{, for $k\in\cb{1,2,\ldots,\tilde N}$}.\nonumber 
\end{align}

As long as the Hermite-Birkhoff interpolation conditions at each distinct point are consistent, there
will be at least one feasible solution to this problem. In fact, there is a feasible sequence $b$ with a finite number of non-zero terms (see Remark \ref{rem:interp} in Subsection \ref{subsec:finite} for a simple upper bound on the minimum number of terms needed). We now show uniqueness.

\begin{proposition}\label{prop:unique}
If the feasible set for (\ref{eq:hbproblem}) is non-empty and
$\norm{\vc\omega}_2^{p}d^{-\frac{1}{2}}$\\$ = \cb{\norm{\vc\omega_n}_2^{p}d_n^{-\frac{1}{2}}}_{n=1}^\infty\in\ell^{2}$,
then the solution
to (\ref{eq:hbproblem}) is unique. \end{proposition}
\begin{proof}
Define the linear operator $\L_{\text{HB}}:\ell^{2}\to\C^{\tilde{N}}$
by: 
\[
\b{\L_{\text{HB}}b}_{k}:=\sum_{n=1}^{\infty}d_{n}^{-\frac{1}{2}}b_{n}\b{\F_{k}\phi_{n}}\b{\x_{k}}.
\]

If $\norm{\vc\omega}_2^{p}d^{-\frac{1}{2}}\in\ell^{2}$, then noting that each $\F_k$ is of order at most $p$ and has bounded coefficient functions, it is clear that $\L_{\text{HB}}$ is bounded, since for $q\in\cb{0,1,\ldots,p}$,
\[
\abs{\sum_{n=1}^{\infty}d_{n}^{-\frac{1}{2}}b_{n}\norm{\vc\omega_n}_2^{q}\phi_{n}\b{\x_{k}}}\le\norm{\norm{\vc\omega}_2^{q}d^{-\frac{1}{2}}}_{\ell^{2}}\norm b_{\ell^{2}}\text{, by Cauchy-Schwarz}.
\]
Since $\L_{\text{HB}}$ is bounded, $\mathcal{N}\b{\L_{\text{HB}}}$
is closed and the constraint set is closed and convex. By a standard
result in functional analysis \citep[see, for example,][3.3-1]{kreys91}, if the constraint set is non-empty, 
there is a unique element with a minimum norm; this is the unique solution. 
\end{proof}

\subsection{Native Hilbert Spaces}

Let $\tilde{b}$ be the solution to (\ref{eq:hbproblem}), then our
chosen Hermite-Birkhoff interpolant is
\begin{equation}
\tilde{u}=\sum_{n=1}^{\infty}d_{n}^{-\frac{1}{2}}\tilde{b}_{n}\phi_{n} \label{eq:hbint}.
\end{equation}

Our chosen solution $\tilde{u}$ turns out to be an element of a particular
Hilbert space of functions constructed from an isometry to $\ell^{2}$.
\begin{definition}  Let $\cb{\phi_{n}}_{n=1}^{\infty}$
be the Fourier basis for the box $\Omega$ with periodic boundary
conditions ($\phi_n\b\x = e^{i\vc{\omega}_n \cdot \x}$), then we define the Hilbert space $\H\b d$ by
\[
\H\b d :=\cb{\sum_{n=1}^{\infty}d_{n}^{-\frac{1}{2}}b_{n}\phi_{n}:b\in\ell^{2}}.
\]
The associated inner product for $\H\b d$ is
\[
\b{\sum_{n=1}^{\infty}d_{n}^{-\frac{1}{2}}a_{n}\phi_{n},\sum_{n=1}^{\infty}d_{n}^{-\frac{1}{2}}b_{n}\phi_{n}}_{d}=\b{a,b}_{\ell_{2}}=\sum_{n=1}^{\infty}d_{n}\b{d_{n}^{-\frac{1}{2}}a_{n}}\b{d_{n}^{-\frac{1}{2}}b_{n}}^{*}.
\]
Alternatively, if $u,v\in\H\b d$ have sequences of Fourier coefficients $\hat{u},\hat{v}$, respectively, then
\begin{align*}
\b{u,v}_{d} & =\sum_{n=1}^{\infty}d_{n}\hat{u}_{n}\hat{v}_{n}^{*}, \quad
\norm u_{d} =\sqrt{\sum_{n=1}^{\infty}d_{n}\abs{\hat{u}_{n}}^{2}}=\norm{d^{\frac{1}{2}}\hat{u}}_{\ell^{2}}.
\end{align*}
\end{definition}

For example, if $d_n=\b{1+\norm{\vc\omega_n}_2^{2}}^q$, then $\H\b d=H^{q}\b{\Omega}$; this would be a form of minimum Sobolev norm interpolation \citep{chand13,chand18}, which has been used for PDEs in flat domains previously in a different setup \citep{chand15}. Note that we often refer to $\norm{\cdot}_d$ as the $d$-norm.

\begin{remark} $\tilde{u}$ is the solution to the constrained optimization
problem: 
\begin{align*}
\text{minimize, over \ensuremath{u\in\H\b d}: } & \norm u_{d}\\
\text{subject to: } & \b{\F_{k}u}\b{\x_{k}}=f_{k} \text{, for each }k\in\cb{1,2,\ldots,\tilde N}.
\end{align*}
\end{remark}

\subsection{The Adjoint Range and Connection to Radial Basis Functions}

\label{subsec:adjrange} 
Recall that we want to match certain derivative conditions at each point $\x_{k}\in S$;
there is some collection of linear differential operators $\cb{\F_{k}}_{k=1}^{\tilde{N}}$ so that
\[
\sum_{n=1}^{\infty}d_{n}^{-\frac{1}{2}}\tilde{b}_{n}\b{\F_{k}\phi_{n}}\b{\vc x_{k}}=f_k\text{, for each $k\in \cb{1,2,\ldots,\tilde N}$},
\]
where $\tilde{b}$ is again the solution to (\ref{eq:hbproblem}).
Then, where $p$ is the maximum order of $\F_{k}$ over all $k\in\cb{1,2,\ldots,\tilde N}$, $\L_\text{HB}:\ell^2 \to \C^{\tilde N}$ is bounded as long as 
$\norm{\vc\omega}_2^{p}d^{-\frac{1}{2}}\in\ell^{2}$ (see Proposition \ref{prop:unique}), and elementary functional analysis gives us: 
\begin{align}
\L_\text{HB} & :b\mapsto\b{\sum_{n=1}^{\infty}d_{n}^{-\frac{1}{2}}\b{\F_{k}\phi_{n}}\b{\x_{k}}b_{n}}_{k=1}^{\tilde{N}}\nonumber\\
\L_\text{HB}^{*} & :\vc{\beta}\mapsto\b{\sum_{k=1}^{\tilde{N}}d_{n}^{-\frac{1}{2}}\b{\F_{k}\phi_{n}}^{*}\b{\x_{k}}\beta_{k}}_{n=1}^{\infty}\nonumber\\
\tilde{b} & \in \mathcal{N}\b{\L_\text{HB}}^\perp = \mathcal{R}\b{\L_\text{HB}^*}=%
\text{span}\cb{\b{d_{n}^{-\frac{1}{2}}\b{\F_{k}\phi_{n}}^{*}\b{\x_{k}}}_{n=1}^{\infty}}_{k=1}^{\tilde{N}}\label{eq:kerperp}.
\end{align}

We then define a set of functions $\cb{\psi_k}_1^{\tilde N}$ by:
\begin{align*}
\psi_{k} & :=\sum_{n=1}^{\infty}d_{n}^{-\frac{1}{2}}\b{d_{n}^{-\frac{1}{2}}\b{\F_{k}\phi_{n}}^{*}\b{\x_{k}}}\phi_{n}.
\end{align*}

Then, there exists some function $\tilde{u}\in\text{span}\cb{\psi_{k}}_{k=1}^{\tilde{N}}$
so that for each $i$,
\[
\b{\F_{i}\tilde{u}}\b{\x_{i}}=f_i.
\]

Furthermore, the Fourier coefficients of $\tilde{u}$ minimize $\sum_{n=1}^{\infty}d_{n}\abs{a_{n}}^{2}$
over all possible choices of Fourier coefficients $a_{n}$ that match
the constraint. Finally, let: $\tilde{u}=\sum_{j=1}^{\tilde{N}}\beta_{j}\psi_{j}$, then we want $\vc\beta=\b{\beta_j}_{j=1}^{\tilde N}\in\C^{\tilde N}$ to solve

\[
\sum_{j=1}^{\tilde{N}}\beta_{j}\b{\sum_{n=1}^{\infty}d_{n}^{-1}\b{\F_{j}\phi_{n}}^{*}\b{\x_{j}}\b{\F_{i}\phi_{n}}\b{\x_{i}}} =
\sum_{j=1}^{\tilde{N}}\vc{\Phi}_{ij}\beta_{j} =f_i,
\]
where
\[
\vc{\Phi}_{ij}=\sum_{n=1}^{\infty}d_{n}^{-1}\b{\F_{j}\phi_{n}}^{*}\b{\vc x_{j}}\b{\F_{i}\phi_{n}}\b{\vc x_{i}}=\vc{\Phi}_{ji}^{*}.
\]

Therefore, the system that must be solved for $\vc\beta$ is self-adjoint. Also, for any $\vc{\alpha}\in\C^{\tilde{N}}$,
\begin{align*}
\vc{\alpha}^{*}\vc{\Phi}\vc{\alpha} & =\sum_{i=1}^{\tilde{N}}\sum_{j=1}^{\tilde{N}}\b{\sum_{n=1}^{\infty}d_{n}^{-1}\b{\F_{j}\phi_{n}}^{*}\b{\x_{j}}\b{\F_{i}\phi_{n}}\b{\x_{i}}}\alpha_{i}^{*}\alpha_{j}\\
 & =\sum_{n=1}^{\infty}\b{\sum_{i=1}^{\tilde{N}}d_{n}^{-\frac{1}{2}}\b{\F_{i}\phi_{n}}\b{\x_{i}}\alpha_{i}^{*}}\b{\sum_{j=1}^{\tilde{N}}d_{n}^{-\frac{1}{2}}\b{\F_{j}\phi_{n}}^{*}\b{\x_{j}}\alpha_{j}}\\
 & =\sum_{n=1}^{\infty}\b{\L_\text{HB}^{*}\vc{\alpha}}_{n}^{*}\b{\L_\text{HB}^{*}\vc{\alpha}}_{n} =\norm{\L_\text{HB}^{*}\vc{\alpha}}_{\ell^{2}}^2 \ge0.
\end{align*}

Thus, $\vc \Phi$ is positive semi-definite. Moreover, $\vc{\alpha}^{*}\vc{\Phi}\vc{\alpha}\ne0$ for $\vc{\alpha\ne0}$ due to the existence
of an interpolant (namely, the norm-minimizing interpolant in $\H \b d$) in $\text{span}\cb{\psi_{k}}_{k=1}^{\tilde{N}}$ for any set of interpolation values $\cb{f_i}_{i=1}^{\tilde N}$. This implies uniqueness of $\vc \beta$ %
since the system is square. That is, for any set of interpolation
conditions $\b{\F_{i}\tilde{u}}\b{\vc x_i}=f_i$ for $i\in \cb{1,2,\ldots,\tilde N}$, there is a
unique, $d$-norm minimizing interpolant in $\text{span}\cb{\psi_{k}}_{k=1}^{\tilde{N}}$
such that the system to find the coefficients of the interpolant in
the $\cb{\psi_{k}}_{k=1}^{\tilde{N}}$ basis is self-adjoint and positive
definite.

Also note that $\b{\psi_{k},\psi_{j}}_{d}=\b{\L_\text{HB}^{*}\vc e_{k},\L_\text{HB}^{*}\vc e_{j}}_{\ell^{2}}=\vc{\Phi}_{kj}$.
This is closely related to symmetric Hermite RBF methods for PDEs
\citep[see][]{sun94,frank98}, which have been applied to problems in flat domains and
rely on a similar norm-minimizing property. 

\subsection{Setting up a PDE}

We now set up a PDE problem. Let $S^{\b j}\subset\Omega$ for $j\in\cb{1,2,\ldots,N_S}$ 
all be manifolds, open bulk domains, or single points, where $\Omega$ is a box-shaped domain in $\R^m$.
Define point clouds $S_{\tilde{N}_{j}}^{\b j}:=\cb{\x_{k}^{\b j}}_{k=1}^{\tilde{N}_{j}}\ss S^{\b j}$.
We consider the problem:
\begin{align}
\F^{\b j}u\eval_{S^{\b j}} & =f^{\b j}\text{, for each \ensuremath{j\in\cb{1,2,\ldots, N_S}}},\label{eq:fullpde}
\end{align}
where each $\F^{\b j}$ is a linear differential
 operator of order at most $p$ and $f^{\b j}$ is defined in a neighbourhood of each point in $S^{\b j}$ for each $j$. 
This form covers all strong-form linear PDEs
on manifolds that we may consider.  For example, to solve a Laplace-Beltrami problem on a surface $S$, we may set $S^{\b{1}}=S^{\b{2}}=S$ and $S^{\b{3}}=\partial S$. We could then choose $\F^{\b{1}}u=\Delta u - \n_S\cdot \b{D^2u}\n_S $, $\F^{\b{2}}u=\nabla u-\n_S\b{\nabla u\cdot\n_S}$, $f^{\b 1}=f$, and $f^{\b 2}=0$ so that $\Delta_Su\eval_S=f$ by Corollary \ref{cor:Let--and}. $\F^{\b 3}$ would determine the boundary conditions.

We then have the discretization:

\begin{align}
\text{minimize, over \ensuremath{b\in\ell^{2}}: } & \norm b_{\ell^{2}}\label{eq:pde1}\\
\text{subject to: } & \sum_{n=1}^{\infty}d_{n}^{-\frac{1}{2}}b_{n}\b{\F^{\b j}\phi_{n}}\b{\x_{k}^{\b j}}=f^{\b j}\b{\x_{k}^{\b j}}\text{, for each \ensuremath{j,k}}.\nonumber 
\end{align}

Assuming $\norm{\vc{\omega}}_{2}^{p}d^{-\frac{1}{2}}\in\ell^{2}$, let $\tilde{u}=\sum_{n=1}^{\infty}d_{n}^{-\frac{1}{2}}\tilde{b}_{n}\phi_{n}$
as before, where $\tilde{b}$ is the unique solution to (\ref{eq:pde1}), and the functions $\cb{\phi_n}$ are again Fourier modes on the box $\Omega$ so that $\phi_n \b \x = e^{i\vc \omega_n \cdot \x}$. If $u$
is a solution to the original PDE (\ref{eq:fullpde}), extended to $\Omega$, and $u\in\H\b d$,
then $u$ is a feasible Hermite-Birkhoff interpolant and $\norm{\tilde{u}}_{d}\le\norm u_{d}$. Note that we can always choose $d$ to satisfy $\norm{\vc{\omega}}_{2}^{p}d^{-\frac{1}{2}}\in\ell^{2}$ to be guaranteed a unique solution to (\ref{eq:pde1}) and have $\norm{\tilde{u}}_{d}\le\norm u_{d}<\infty$ as long as an exact solution $u$ is sufficiently smooth. For certain surface PDEs involving the Laplace-Beltrami operator with sufficiently smooth solutions on the surface, an extended, periodic solution $u$ on $\Omega$ with the same number of continuous derivatives will exist due to Proposition \ref{prop-cpexist}.

\subsection{Analysis for Hermite-Birkhoff Problem}\label{subsec:analysis}

Let $h_{\text{max}}$ be the largest fill distance of each point cloud
$S_{\tilde{N}_{j}}^{\b j}$ in
their respective domains $S^{\b j}$. We can
then prove uniform convergence of $\F^{\b j}\tilde{u} \to f^{\b{j}}$; convergence is high-order or super-algebraic with respect to $h_\text{max}$ as long as certain smoothness requirements are met. This, or similar arguments with different norms, is enough to prove high-order convergence of $\tilde{u}$ to a solution to the PDE if the PDE is well-posed with respect to the norms in consideration; this is considered in Proposition \ref{prop:conv}. Proposition \ref{prop:Let--be2} will later show convergence to a solution of the PDE directly under weaker conditions, albeit without rate estimates.

Note for the following results that we typically omit the notation for the restriction of a function when the restricted domain is otherwise clear, such as in a norm. For example, we may write $\norm{\tilde u}_{L^\infty\b{S^{\b j}}}$ to mean $\norm{\tilde u\vert_{S^{\b j}}}_{L^\infty \b{S^{\b j}}}$. Noting that if $S^{\b j}$ is a point, then $\F^{\b j}\tilde u = f^{\b j}$ on that point as long as $S^{\b j}_{\tilde N_j}$ is non-empty, we now cover the case that $S^{\b j}$ has dimension at least one.

\begin{proposition}\label{prop:gamma}Let
$\tilde{b}$ be the solution to (\ref{eq:pde1}) and let $\tilde{u}=\sum_{n=1}^{\infty}d_{n}^{-\frac{1}{2}}\tilde{b}_{n}\phi_{n}$. Assume $S^{\b j}$ is of dimension $m_j\ge1$ and can
be parametrized by a finite atlas $\cb{\sigma^{\b j}_k}_{k=1}^{M}$ where
$\sigma^{\b j}_k:U_{k}^{\b j}\to S^{\b j}_k$ and each $U_{k}^{\b j}\subset\R^{m_j}$
satisfies an interior cone condition. Let $q\in\N$ and assume that each $\sigma_k^{\b j}$ is $C^q$
with bounded derivatives up to order $q$. Also, assume that each $\sigma_k^{\b j}$ has an inverse metric tensor with a bounded norm on $\sigma_k^{\b j}\b{U_k^{\b j}}$. Let $j\in\cb{1,2,\ldots,N_S}$ be fixed. Assume $f^{\b j}$ has bounded derivatives up to order $q$ and $F^{(j,q)}d^{-\frac{1}{2}}\in\ell^2$, where $F^{(j,q)}:=\cb{\max_{\abs{\beta}\le q}\norm{\partial^{\beta}\F^{\b j}\phi_n}_{L^\infty \b {S^{\b j}}}}_{n=1}^\infty$. Then, for small enough $h_{\text{max}}$, there exist constants $\tilde{B}_{j,k,q,\abs{\alpha}}>0$ such that
\[
\norm{\partial^\alpha\b{\b{\F^{\b j}\tilde{u}-f^{\b j}}\circ \sigma_k^{\b j}}}_{L^{\infty}\b{U_k^{\b j}}}\le\tilde{B}_{j,k,q,\abs{\alpha}}h_{\text{max}}^{q-\frac{m_j}{2}-\abs\alpha}\norm u_{d},
\]
where $u=\sum_{n=1}^{\infty}d_{n}^{-\frac{1}{2}}{b}_{n}\phi_{n}\in\H\b d$ is a solution to (\ref{eq:fullpde}), $\tilde C_{j,q}:=\norm{F^{\b{j,q}}d^{-\frac{1}{2}}}_{\ell^2}$, and $\alpha$ is a multi-index with $\abs{\alpha}\in\cb{0,1,2,\ldots,p}$.

\end{proposition}
\begin{proof}
Since the derivatives of $\sigma^{\b j}_k$ are bounded, repeated application
of chain rule shows there exist constants $A_{k,q}$ so that for each
multi-index $\gamma$ of order less than or equal to $q$,
\begin{align*}
& \ \norm{\partial^{\gamma}\b{\F^{\b j}\tilde{u}\circ\sigma^{\b j}_k-f^{\b j}\circ\sigma^{\b j}_k}}_{L^{\infty}\b{U^{(j)}_{k}}} \\
\le& \ A_{k,q}\max_{\abs{\beta}\le q}\norm{\partial^{\beta}\b{\F^{\b j}\tilde{u}-f^{\b j}}}_{L^{\infty}\b{S^{\b j}}}\\
=& \ A_{k,q}\max_{\abs{\beta}\le q}\norm{\partial^{\beta}\b{\F^{\b j}\b{\tilde{u}-u}}}_{L^{\infty}\b{S^{\b j}}}\\
 \le & \ A_{k,q}\max_{\abs{\beta}\le q}\b{\sum_{n=1}^{\infty}d_{n}^{-\frac{1}{2}}\abs{\tilde{b}_n - b_n}\norm{\partial^\beta\F^{\b j}\phi_{n}}_{L^{\infty}\b{S^{\b j}}}
}\\
 \le & \ A_{k,q}\tilde C_{j,q}\norm{\tilde{b} - b}_{\ell^{2}}\text{, by Cauchy-Schwarz}.
\end{align*}

Now we note that $b$ is feasible for (\ref{eq:pde1}), so $\tilde{b}-b\in\mathcal{N}\b{\L_\text{HB}}$ in the notation of Subsection \ref{subsec:adjrange}, so $\tilde{b}-b\perp\tilde b$ by (\ref{eq:kerperp}). Importantly, $\norm{\tilde{b}-b}_{\ell_2}=\sqrt{\norm{b}_{\ell_2}^2-\norm{\tilde b}^2_{\ell_2}}\le\norm{b}_{\ell_2}$.
Then there exist constants $\tilde A_{j,k,q}>0$ so that
\[
\norm{\b{\F^{\b j}\tilde{u}-f^{\b j}}\circ\sigma^{\b j}_k}_{H^{q}\b{U_{k}^{\b j}}}\le \tilde A_{j,k,q}\norm{b}_{\ell^{2}}.
\]

Once again, using Corollary \ref{cor:Let--with}, we have that
there exist constants $\tilde{B}_{k,q,\abs{\alpha},j}>0$ so that, for small enough $h_\text{max}$,
\[
\norm{\partial^\alpha\b{\b{\F^{\b j}\tilde{u}-f^{\b j}}\circ \sigma_k^{\b j}}}_{L^{\infty}\b{U_k^{\b j}}}\le\tilde{B}_{j,k,q,\abs{\alpha}}h_{\text{max}}^{q-\frac{m_j}{2}-\abs\alpha}\norm u_{d}.
\]

Note that Corollary \ref{cor:Let--with} is applied to $U^{\b j}_k$, not $S^{\b j}_k$, so we are implicitly relying on the boundedness of the inverse metric tensor of $\sigma^{\b j}_k$ so that the fill distance on $U^{\b j}_k$ is at most a constant multiple of the fill distance on $S^{\b j}_k$. 
\end{proof}

In particular, notice that we can let $\tilde{B}_{j,q}=\max\cb{\tilde{B}_{j,k,q,0}}_{k=1}^{M}$,
then
\[
\norm{\F^{\b j}\tilde{u}-f^{\b j}}_{L^{\infty}\b{S^{\b j}}}\le\tilde{B}_{j,q}h_{\text{max}}^{q-\frac{m_j}{2}}\norm{u}_{d}.
\]

We note that the $F^{\b{j,q}}d^{-\frac{1}{2}}\in\ell^2$ condition tells us how quickly our choice of $d^{-\frac{1}{2}}$ must decay. If $\F^{\b j}$ is of order $p$ and has bounded coefficient functions with $q$ bounded derivatives, then $F^{\b{j,q}}$ will grow as $\norm{\vc \omega_n}_2^{p+q}$ as $n\to\infty$.

\subsection{Finite Number of Basis Functions}\label{subsec:finite}

We now truncate our Fourier series to a finite basis of $N_{b}$ terms. Let $\L_{N_{b}}:=\L_\text{HB}\eval_{\C^{N_{b}}}:\C^{N_b}\to\C^{\tilde N}$, where $\L_{\text{HB}}$ is as defined in the proof of Proposition \ref{prop:unique}. Let $\vc f:=\b{f^{\b j}\b{\x_{k}^{\b j}}}$
be a vectorized form of all Hermite-Birkhoff data
from (\ref{eq:pde1}) (a column vector using all values of $j$ and $k$). The problem then becomes:

\begin{align}\label{eq:finitebasis}
\text{minimize, for }\vc b\in\C^{N_{b}}: & \norm{\vc b}_{2}\\
\text{subject to: } & \L_{N_{b}}\vc b=\text{proj}_{\mathcal{R}\b{\L_{N_{b}}}}
\vc f\nonumber.
\end{align}

If a Hermite-Birkhoff interpolant exists in $\text{span}\cb{\phi_{n}}_{n=1}^{N_{b}}$,
then

\[
\text{proj}_{\mathcal{R}\b{\L_{N_{b}}}}
\vc f=\vc f.
\]

Let $\tilde{\vc b}$ be the solution to this optimization problem. 
\begin{proposition} \label{prop:Let--and}Assume there is a constant $B_p$ such that for all large enough $n$, $\norm{\F^{\b j}\phi_n}_{L^\infty\b{S^{\b j}}}\le B_p\norm{\vc{\omega}_n}_2^p$ for each $j\in\cb{1,2,\ldots,N_S}$. Also assume $\norm{\vc\omega}_2^{p}d^{-\frac{1}{2}}=\cb{\norm{\vc\omega_n}_2^{p}d_n^{-\frac{1}{2}}}_{n=1}^\infty\in\ell^{2}$. Then, for a fixed number of total Hermite-Birkhoff interpolation conditions $\tilde{N}$, there exists a constant $A_{\tilde{N}}>0$
such that for sufficiently large $N_{b}$,
\[
\norm{\tilde{u}_{\infty}-\tilde{u}_{N_{b}}}_{L^{\infty}\b \Omega}\le A_{\tilde{N}}\b{\sum_{n=N_{b}+1}^{\infty}\norm{\vc{\omega}_{n}}_{2}^{2p}d_{n}^{-1}}^{\frac{1}{2}},
\]
where $\tilde{u}_{\infty}:=\sum_{n=1}^{N_b} d_n^{-\frac{1}{2}} \tilde b_{\infty,n}  \phi_n\in\H\b d$ is from the solution $\tilde b_\infty \in \ell^2$ to
the full $\ell^{2}$ problem (\ref{eq:pde1}) and $\tilde{u}_{N_{b}}$ is the finite
basis solution: $\tilde u_{N_b} = \sum_{n=1}^{N_b} d_n^{-\frac{1}{2}} \tilde b_n \phi_n$, where $\tilde b_n$ is from the solution $\tilde{\vc b}$ to (\ref{eq:finitebasis}). \end{proposition}

The approach of the proof is to note that if an interpolant exists
for finite $N_{b}$, it is also feasible for the $\ell^{2}$ problem,
while the truncated solution to the $\ell^{2}$ problem will “nearly”
be an interpolant and will be close to a feasible solution for the
finite basis problem. Along with orthogonality of the minimizer to
the null space, this turns out to be sufficient to show convergence proportional
to $\b{\sum_{n=N_{b}+1}^{\infty}\norm{\vc{\omega}_{n}}_{2}^{2p}d_{n}^{-1}}^{\frac{1}{2}}$
of the truncated basis solution to the $\ell^{2}$ problem. The idea here is to show that adding terms to our Fourier basis recovers the solution to the full $\ell^2$ problem in the limit $N_b\to\infty$.

\begin{proof}
To start, let $\tilde{N}$ be fixed. Let
$\vc V_{N_{b}}$ be the matrix corresponding
to $\L_{N_{b}}$, so that
\begin{align*}
\vc V_{N_{b}}\tilde{\vc{b}} & =\vc f,
\end{align*}
then the unique solution to the finite-dimensional problem is
\[
\tilde{\vc b}=\vc V_{N_{b}}^{+}\vc f,
\]
where $^{+}$ indicates the Moore-Penrose pseudoinverse. Now, let
$N_{b,\text{interp}}$ be the smallest number of basis functions for
which $\L_{N_{b}}$ is surjective (such a finite value always exists, see Remark \ref{rem:interp}). Note that if
$\tilde{\vc b}_{N_{b,\text{interp}}}$ is the solution to the optimization
problem for $N_{b,\text{interp}}$ basis functions, then $\tilde{\vc b}_{N_{b,\text{interp}}}$
padded with zeros to be in $\C^{N_{b}}$ is feasible for any $N_{b}>N_{b,\text{interp}}$.
Therefore, 
\[
\norm{\tilde{\vc b}}_{2}\le\norm{\tilde{\vc b}_{N_{b,\text{interp}}}}_{2}\text{, for any }N_{b}>N_{b,\text{interp}}.
\]

This holds for any $\vc f$, so if we define $M_{\tilde{N}}=\norm{\vc V_{N_{b,\text{interp}}}^{+}}_{2}$
, then as long as $N_{b}\ge N_{b,\text{interp}}$,
\[
\norm{\vc V_{N_{b}}^{+}}_{2}\le M_{\tilde{N}}.
\]

In particular, $\norm{\vc V_{N_{b}}^{+}}_{2}$ is bounded
as $N_{b}\to\infty$. Now define $\tilde{b}_{\infty}$ as the solution
to the full $\ell^{2}$ problem, and define the truncated $\tilde{\vc b}_{\infty}^{\b{N_{b}}}=\b{\tilde{b}_{\infty,n}}_{n=1}^{N_{b}}$.
Then, define
\[
\lc_{\text{trunc}}:=\norm{\vc f-\vc V_{N_{b}}\tilde{\vc b}_{\infty}^{\b{N_{b}}}}_{2},
\]
and note that
\[
\vc V_{N_{b}}\b{\tilde{\vc b}_{\infty}^{\b{N_{b}}}+\vc V_{N_{b}}^{+}\b{\vc f-\vc V_{N_{b}}\tilde{\vc b}_{\infty}^{\b{N_{b}}}}}=\vc f.
\]

Therefore, since $\tilde{\vc b}$ is the minimizer,
\begin{align}
\norm{\tilde{\vc b}_{\infty}^{\b{N_{b}}}+\vc V_{N_{b}}^{+}\b{\vc f-\vc V_{N_{b}}\tilde{\vc b}_{\infty}^{\b{N_{b}}}}}_{2} & \ge\norm{\tilde{\vc b}}_{2}\nonumber \\
\implies M_{\tilde{N}}\lc_{\text{trunc}} \ge\norm{\tilde{\vc b}}_{2}-\norm{\tilde{\vc b}_{\infty}^{\b{N_{b}}}}_{2}
&\ge \norm{\tilde{\vc b}}_{2}-\norm{\tilde{b}_{\infty}}_{\ell^{2}} \ge 0,\label{eq:trunc}
\end{align}
where we note $\norm{\tilde{\vc b}}_{2}\ge\norm{\tilde{b}_{\infty}}_{\ell^{2}}\ge\norm{\tilde{\vc b}_{\infty}^{\b{N_{b}}}}_{2}$
since $\tilde{\vc b}$ padded with zeros is feasible for the full
$\ell^{2}$ problem. More precisely, let $\mathcal{E}:\C^{N_{b}}\to\ell^{2}$ so that
$\b{\mathcal{E}\vc b}_{n}=b_{n}$ for $n\le N_{b}$ and zero otherwise,
then $\mathcal{E}\tilde{\vc b}$ is feasible for the $\ell^{2}$
problem and $\norm{\tilde{\vc b}}_{2}=\norm{\mathcal{E}\tilde b}_{\ell^2}\ge\norm{\tilde{b}_{\infty}}_{\ell^{2}}$. $\tilde{b}_{\infty}$ is the minimizer and is therefore orthogonal
to the null space of $\L_\text{HB}$, so
$\norm{\tilde{b}_{\infty}}_{\ell^{2}}^{2} =\b{\tilde{b}_{\infty},\mathcal{E}\tilde{\vc b}}_{\ell^{2}}$ and

\begin{align*}
\norm{\tilde{b}_{\infty}-\mathcal{E}\tilde{\vc b}}_{\ell^{2}}^{2}  %
  =\norm{\tilde{\vc b}}_{2}^{2}-\norm{\tilde{b}_{\infty}}_{\ell^{2}}^{2}
  \le\b{\norm{\tilde{b}_{\infty}}_{\ell^{2}}+\norm{\tilde{\vc b}}_{2}}M_{\tilde{N}}\lc_{\text{trunc}},
\end{align*}
where we use equation (\ref{eq:trunc}). Then,

\begin{align*}
\norm{\tilde{b}_{\infty}}_{\ell^{2}}+\norm{\tilde{\vc b}}_{2} & =\norm{\tilde{b}_{\infty}}_{\ell^{2}}+\norm{\tilde{\vc b}}_{2}-\norm{\tilde{b}_{\infty}}_{\ell^{2}}+\norm{\tilde{b}_{\infty}}_{\ell^{2}}\\
 & \le2\norm{\tilde{b}_{\infty}}_{\ell^{2}}+M_{\tilde{N}}\lc_{\text{trunc}}\\
 \implies \norm{\tilde{b}_{\infty}-\mathcal{E}\tilde{\vc b}}_{\ell^{2}}^{2} &\le\b{2\norm{\tilde{b}_{\infty}}_{\ell^{2}}+M_{\tilde{N}}\lc_{\text{trunc}}}M_{\tilde{N}}\lc_{\text{trunc}},
\end{align*}
where we again use (\ref{eq:trunc}). Now, using the line above and Cauchy-Schwarz,
\begin{align}
\norm{\tilde{u}_{\infty}-\tilde{u}_{N_{b}}}_{L^{\infty}\b \Omega} & =\norm{\b{\sum_{n=1}^{\infty}d_{n}^{-\frac{1}{2}}\tilde{b}_{\infty,n}\phi_{n}}-\b{\sum_{n=1}^{N_{b}}d_{n}^{-\frac{1}{2}}\tilde{b}_{n}\phi_{n}}}_{L^{\infty}\b \Omega}\nonumber \\
 & \le\b{\abs{\tilde{b}_{\infty}-\mathcal{E}\tilde{\vc b}},d_{n}^{-\frac{1}{2}}}_{\ell^2}\nonumber \\
 & \le\norm{d^{-\frac{1}{2}}}_{\ell^{2}}\sqrt{\b{2\norm{\tilde{b}_{\infty}}_{\ell^{2}}+M_{\tilde{N}}\lc_{\text{trunc}}}M_{\tilde{N}}\lc_{\text{trunc}}}.\label{eq:diff}
\end{align}

Also,
\begin{align*}
\lc_{\text{trunc}}^2 & = \sum_{j=1}^{N_S} \sum_{k=1}^{\tilde{N}_j}\abs{f^{\b j}\b{\x^{\b{j}}_{k}}-\sum_{n=1}^{N_{b}}d_{n}^{-\frac{1}{2}}\tilde{b}_{\infty,n}\F^{\b j}\phi_{n}\b{\x^{\b j}_{k}}}^{2}\\
 & =\sum_{j=1}^{N_S} \sum_{k=1}^{\tilde{N}_j}\abs{\sum_{n=N_b + 1}^{\infty}d_{n}^{-\frac{1}{2}}\tilde{b}_{\infty,n}\F^{\b j}\phi_{n}\b{\x^{\b j}_{k}}}^{2} \\
 & \le B_p^2\tilde{N}\b{\sum_{n=N_{b}+1}^{\infty}\abs{\tilde{b}_{\infty,n}}^{2}}\b{\sum_{n=N_{b}+1}^{\infty}\norm{\vc \omega_n}_2^{2p} d_{n}^{-1}}\text{, by Cauchy-Schwarz},
\end{align*}
where the last line holds for large enough $N_b$ by assumption on $\F^{\b j}$. Finally, from (\ref{eq:diff}),
\begin{align*}
 & \norm{\tilde{u}_{\infty}-\tilde{u}_{N_{b}}}_{L^{\infty}\b \Omega}\\
\le & \norm{d^{-\frac{1}{2}}}_{\ell^{2}}\b{2\norm{\tilde{b}_{\infty}}_{\ell^{2}}+M_{\tilde{N}}B_p\sqrt{\tilde{N}}\b{\sum_{n=N_{b}+1}^{\infty}\abs{\tilde{b}_{\infty,n}}^{2}}^{\frac{1}{2}}\b{\sum_{n=N_{b}+1}^{\infty}\norm{\vc \omega_n}_2^{2p}d_{n}^{-1}}^{\frac{1}{2}}}^{\frac{1}{2}}\\
 & \quad\cdot\b{M_{\tilde{N}}B_p\sqrt{\tilde{N}}\b{\sum_{n=N_{b}+1}^{\infty}\abs{\tilde{b}_{\infty,n}}^{2}}^{\frac{1}{2}}\b{\sum_{n=N_{b}+1}^{\infty}\norm{\vc \omega_n}_2^{2p}d_{n}^{-1}}^{\frac{1}{2}}}^{\frac{1}{2}}.
\end{align*}

Now, from equation (\ref{eq:kerperp}) in Subsection \ref{subsec:adjrange}, we recall $\tilde{b}_{\infty}\in\mathcal{R}\b{\L_\text{HB}^{*}}$.
This means $\abs{\tilde{b}_{\infty,n}}=\O\b{\norm{\vc{\omega}_{n}}_{2}^{p}d_{n}^{-\frac{1}{2}}}$ as $n\to\infty$.
From the bound above, we can conclude that for large enough $N_{b}$,
there exists some constant $A_{\tilde{N}}$ such that
\[
\norm{\tilde{u}_{\infty}-\tilde{u}_{N_{b}}}_{L^{\infty}\b \Omega}\le A_{\tilde{N}}\b{\sum_{n=N_{b}+1}^{\infty}\norm{\vc{\omega}_{n}}_{2}^{2p}d_{n}^{-1}}^{\frac{1}{2}}.
\]
That is, $\tilde{u}_{N_{b}}\to\tilde{u}_{\infty}$ uniformly on $\Omega$ as long as $\norm{\vc{\omega}}_{2}^{p}d^{-\frac{1}{2}}\in\ell^{2}$.
\end{proof}

The $A_{\tilde{N}}$ coefficient of the Proposition \ref{prop:Let--and}
could be problematic in theory, but in practice, we can make $\b{\sum_{n=N_{b}+1}^{\infty}\norm{\vc{\omega}_{n}}_{2}^{2p}d_{n}^{-1}}^{\frac{1}{2}}$ converge
arbitrarily quickly in $N_{b}$ with a suitable choice of $d_n$, so
any highly underdetermined system will produce reasonable results. Furthermore, $M_{\tilde N}$, which is a term in $A_{\tilde{N}}$, could be replaced by $\norm{\vc V_{N_b}^+}_2$, which is non-increasing and will typically decrease as $N_b$ increases.

\begin{remark}\label{rem:interp}
In 1D, an interpolant for a combination of function or derivative
interpolation conditions (a Hermite-Birkhoff interpolant) exists under certain conditions when $N_{b}=\tilde{N}$ \citep{johns75}, but this
is not necessarily the case in dimensions greater than one, again due
to the Mairhuber-Curtis Theorem \citep{curti59,mairh56}. An interpolant
does, however, always exist with some finite number of basis functions.

In particular, it can be quickly shown that $N_{b,\text{interp}}\le \b{\b{p+1}\tilde N_\text{distinct}+1}^m$, where $\tilde N_\text{distinct}\le\tilde{N}$ is the number of distinct points in the box $\Omega\subset\R^m$. $\tilde N_\text{distinct}$ can be less than $\tilde{N}$ if one or more points have two or more interpolation conditions imposed. For $m=1$, it has been known for some time \citep[Thm. 1]{johns75} that we can always construct a trigonometric polynomial with at most $\b{p+1}\tilde N_\text{distinct}+1$ terms that is zero and has all derivatives of up to order $p$ equal to zero on up to $\tilde N_\text{distinct}$ distinct points of our choosing, except for one chosen point where either the function or one of its derivatives are equal to one. For $m>1$, we can take products of these 1D functions to get a trigonometric polynomial with at most $\b{\b{p+1}\tilde N_\text{distinct}+1}^m$ terms that is zero and has derivatives equal to zero on a tensor grid of $\tilde{N}_\text{distinct}^m$ points (that our desired $\tilde{N}_\text{distinct}$ points are a subset of), except for one point and derivative of interest. A linear combination of these functions is a feasible interpolant. Experimentally, however, we typically observe $N_{b,\text{interp}}$ to simply be the total number of interpolation conditions ($\tilde{N}$) or not much larger (less than $2\tilde{N}$).

Theorem 2.2 of \cite{chand18} provides an upper bound on $N_{b,\text{interp}}$ that depends on point separation and will often be less than the simple bound presented above. Note that increasing $N_b$ well beyond $N_{b,\text{interp}}$ can improve the error, in accordance with Proposition \ref{prop:Let--and}.

\end{remark}

\subsection{Boundedness, Solvability, and Convergence\label{subsec:A-Note-on}}

In many applications, we are not only interested in finding solutions to a PDE; we may also want to first determine the values of parameters for which the PDE has a solution at all. These are existence problems, which can often be difficult to
analyze numerically. Of particular interest in numerical analysis are eigenvalue problems, where we can find the eigenvalues of $\mathcal{F}$ by finding
the values of $\lambda\in\C$ for which $\b{\mathcal{F}-\lambda}u=0$
has a non-zero solution.

Consider, for some linear differential operators $\mathcal{F,G}$ of order at most
$p$, the PDE:
\begin{equation}\label{eq:fullprob}
\mathcal{F}u =f\text{, on }S,\quad 
\mathcal{G}u\eval_{\partial S} =g.
\end{equation}

To discretize this problem, let $S_{\tilde{N}}:=\cb{\x_{j}}_{j=1}^{\tilde{N}}\subset S$ and
$\partial S_{\tilde{N}_{\partial}}:=\cb{\y_{j}}_{j=1}^{\tilde{N}_{\partial}}\subset\partial S$ be sets of points. Then, let $\vc a \in S$ be fixed. The discretized problem is then
\begin{align}
\text{minimize, over \ensuremath{u^{\b{\tilde{N}}}\in\H\b d}: } & \norm{u^{\b{\tilde{N}}}}_{d}\label{eq:finiteprob}\\
\text{subject to: } & \mathcal{F}u^{\b{\tilde{N}}}\eval_{S_{\tilde{N}}}=f\nonumber,\
\ \mathcal{G}u^{\b{\tilde{N}}}\eval_{\partial S_{\tilde{N}_{\partial}}}=g\nonumber,\
 u^{\b{\tilde{N}}}\b{\vc a}=1,\nonumber 
\end{align}
where we only include the $u^{\b{\tilde{N}}}\b{\vc a}=1$ condition in the case $f=g=0$ to ensure we get a non-zero solution.
If there is a non-zero solution $u$ to the original PDE in $\H\b d$ (possibly requiring $u\b{\vc a}=1$ as well),
then $u$ is feasible, and $\norm{u^{\b{\tilde{N}}}}_{d}\le\norm u_{d}$
for all $\tilde{N}$. We have shown convergence of $\mathcal{F}u^{\b{\tilde{N}}}\to \mathcal{F}u$
in this case (Proposition \ref{prop:gamma}). However, particularly for eigenvalue
problems, we may also be interested in sufficient conditions for
a solution to exist in $\H\b d$, and for $u^{\b{\tilde N}}$ to converge to a solution. Note that if $\F$ can be written $\F^{\b 1}+c\F^{\b 2}$ for some function $c$, and we instead impose $\F^{\b 1}=f$ and $\F^{\b 2}=0$ on $S_{\tilde N}$, the following analysis does not change; this is typically how we write the Laplace-Beltrami operator, using Corollary \ref{cor:Let--and}.

In Theorem 7.1 of the work by \cite{chand18}, the authors prove a result that, in our notation, would show $\F u\to f$ and $\mathcal{G} u \to g$ pointwise when a solution to (\ref{eq:fullprob}) exists. Our next result adds to this for our approach by noting that mere boundedness of $u^{\b{\tilde{N}}}$ (which would be implied by feasibility) is enough to imply that a solution to (\ref{eq:fullprob}) must exist, and that $u^{\b{\tilde{N}}}$ will converge to a solution $u$ in the $d$-norm. For choices of $d$ such that $\norm{\vc{\omega}}_2^{p}d^{-\frac{1}{2}}\in\ell^{2}$, this implies uniform convergence of $u^{\b{\tilde{N}}}$ and its derivatives of up to order $p$ to $u$ over all of $\Omega$. %

\begin{proposition}
\label{prop:Let--be2}Let $\mathcal{F,G}$ be $p$-th or lower order linear
differential operators with bounded coefficient functions, and assume $\norm{\vc{\omega}}_2^{p}d^{-\frac{1}{2}}\in\ell^{2}$.
Let $S_{1}\subset S_{2}\subset\cdots\subset S$, where $S_{n}=\cb{\x_{j}}_{j=1}^{\tilde{N}_{n}}$ such that
$\cb{\tilde{N}_{n}}$ is strictly increasing. Let $\partial S_{1}\subset\partial S_{2}\subset\cdots\subset\partial S$,
where $\partial S_{n}=\cb{\y_{j}}_{j=1}^{\tilde{N}_{\partial,n}}$ such that
$\cb{\tilde{N}_{\partial,n}}$ is strictly increasing. Let $u^{\b n}$
be the corresponding solution to (\ref{eq:finiteprob}) on $S_n$ and $\partial S_n$. If, for all $n$, $\norm{u^{\b n}}_{d}\le\tilde{Q}$
for some $\tilde{Q}>0$, and $h_{\text{max}}^{\b n}\to0$, where $h_{\text{max}}^{\b n}$
is the fill distance associated with $S_{n}$ and $\partial S_{n}$,
then there exists $u\in\H\b d$ so that $u\eval_{S}$ is a solution
to (\ref{eq:fullprob}) and $u^{\b n}\to u$ in $\H\b d$ and all derivatives of $u^{\b n}$ up to order $p$ converge uniformly to the derivatives of $u$
on $\Omega$.
\end{proposition}

\begin{proof}
First, note that for all $m>n$, $u^{\b m}$ is feasible for (\ref{eq:finiteprob})
on $S_{n},\partial S_{n}$, since $S_{n}\subset S_{m},\partial S_{n}\subset\partial S_{m}$.
The fact that $u^{\b n}$ is the solution with minimum $d$-norm implies
$\cb{\norm{u^{\b n}}_d}_n$ is a non-decreasing sequence. $\cb{\norm{u^{\b n}}_d}_n$
is bounded above by $\tilde{Q}$, and therefore $\norm{u^{\b n}}_{d}\uparrow Q$
for some $Q>0$ ($u^{\b n}$ cannot be zero since either $u^{\b n}\b{\vc a}=1$ or one of $f,g$ is non-zero,
and $u^{\b n}$ is $C^p$ since $\norm{\vc{\omega}}_2^{p}d^{-\frac{1}{2}}\in\ell^{2}$
implies uniform convergence of the $u^{\b n}$ Fourier series and its derivatives of up to order $p$). 

Let $\lc>0$. Then, since, $\norm{u^{\b n}}_{d}\uparrow Q$, there exists some $M>0$ so that for all
$n,m>M$ such that $n<m$,
\begin{equation}
0\le\norm{u^{\b m}}_{d}-\norm{u^{\b n}}_{d}<\frac{\lc^{2}}{2Q}.\label{eq:normdiff}
\end{equation}

Now, recall that $u^{\b m}$ is feasible for (\ref{eq:finiteprob}) on $S_{n},\partial S_{n}$,
since $S_{n}\subset S_{m}$ and $\partial S_{n}\subset\partial S_{m}$, so
\begin{align}
\b{u^{\b m}-u^{\b n},u^{\b n}}_{d} & =0\text{, since \ensuremath{u^{\b n}}is the minimal solution}\nonumber \\
\implies\b{u^{\b m},u^{\b n}}_{d}=\b{u^{\b n},u^{\b m}}_{d} & =\norm{u^{\b n}}_{d}^{2}.\label{eq:inner}
\end{align}

Using (\ref{eq:inner}) and then (\ref{eq:normdiff}), it can then be quickly shown that
\begin{align*}
\norm{u^{\b n}-u^{\b m}}_{d} &\le \lc.
\end{align*}

Therefore, $\cb{u^{\b n}}$ is Cauchy and must converge to some $u$
in the $d$-norm since $\mathcal{H}\b d$ is a Hilbert space. Let $\hat{u}^{\b n}$ and $\hat{u}$ be the Fourier
coefficients of $u^{\b n}$ and $u$, respectively. Then, for all $q\le p$, and for any $q$-th order multi-index $\alpha$,
\begin{align*}
\norm{\partial^{\alpha}\b{u^{\b n}-u}}_{L^{\infty}\b{\Omega}} %
 &\le \sum_{j=1}^{\infty}\norm{\vc{\omega}_{j}}_2^{q}d_j^{-\frac{1}{2}}d_j^{\frac{1}{2}}\abs{\hat{u}_{j}^{\b n}-\hat{u}_{j}}\\
 & \le \norm{\norm{\vc{\omega}}_2^{p}d^{-\frac{1}{2}}}_{\ell^{2}}\norm{u^{\b n}-u}_{d} \text{, by Cauchy-Schwarz}\\
 & \to0\text{, as \ensuremath{n\to\infty}}.
\end{align*}

That is, partial derivatives of $u^{\b n}$ up to order $p$ converge
uniformly to the partial derivatives of $u$. This implies $\mathcal{F}u=f$
on $\bigcup_{n=1}^{\infty}S_{n}$. Then, since $h_{\text{max}}^{\b n}\to0$,
$\bigcup_{n=1}^{\infty}S_{n}$ is dense in $S$, and by continuity
of the partial derivatives of $u$ up to order $p$ (again implied
by $\norm{\vc{\omega}}_2^{p}d^{-\frac{1}{2}}\in\ell^{2}$), $\mathcal{F}u=f$
on $S$. A similar argument can be applied on the boundary to show
$\mathcal{G}u\eval_{\partial S}=g$. Therefore, $u$ solves (\ref{eq:fullprob})
and $u^{\b n}\to u$ in the $d$-norm (and uniformly in all derivatives
up to order $p$).
\end{proof}

 Note that we could impose additional conditions on additional domains and reach a similar result. 
 Also notice that if $\mathcal{F}=\mathcal{K}-\lambda$ and $g=0$, then $u$
is an eigenfunction of $\mathcal{K}$ with eigenvalue $\lambda$, subject to the given boundary condition.

A particular application of Proposition \ref{prop:Let--be2} is in
investigating solvability. Informally, Proposition \ref{prop:Let--be2}
shows:
\[
\text{``\ensuremath{d}-norm of \ensuremath{u^{\b n}} bounded"}\implies\text{``Solution to PDE exists in \ensuremath{\H\b d}"},
\]
and we already know:
\[
\text{``Solution to PDE exists in \ensuremath{\H\b d}"}\implies\text{``\ensuremath{d}-norm of \ensuremath{u^{\b n}} bounded"}.
\]
We can then conclude that there is an equivalence:
\[
\text{``\ensuremath{d}-norm of \ensuremath{u^{\b n}} bounded"}\iff\text{``Solution to PDE exists in \ensuremath{\H\b d}".}
\]
This means that for a dense enough point cloud, the $d$-norm of interpolant
solutions to solvable PDEs will be much smaller than the $d$-norm
for solutions for PDEs without a solution; if the PDE is not solvable, the norm must be unbounded. We test this numerically
in Subsection \ref{subsec-eigs} to search for
eigenvalues.

Proposition \ref{prop:Let--be2} shows convergence of $u^{\b n}$ to a solution of the PDE (\ref{eq:fullprob}) without a rate estimate.  To obtain a rate estimate, we need our PDE to have a unique solution and satisfy certain stability or regularity properties. In this case, Proposition \ref{prop:gamma} provides a straightforward way to prove high-order convergence of $u^{\b{n}}\to u$. We prove convergence here for operators satisfying a regularity condition; this condition is satisfied by elliptic operators under certain conditions.

\begin{proposition}\label{prop:conv}
Suppose $S$ is bounded and that for all $f\in C^q\b{\overline S},g\in C^q\b{\partial S}$ (if $S$ has boundary) there is a unique solution $u\in H^r(S)$ to (\ref{eq:fullprob}) for some $r>0$ that can be extended to $\H\b d$. Also, assume that a regularity condition holds for (\ref{eq:fullprob}) such that there exist constants $A,B>0$ and $\tilde p\ge0$ such that for each $v\in H^r\b S$
\begin{align}\label{eq:regular}
\norm{v}_{H^{r}\b S} \le A\norm{\F v}_{L^2\b S} + B\norm{\mathcal{G} v}_{H^{\tilde p}\b{\partial S}}
\end{align}
Then, let $u^{\b {\tilde{N}}}$ be the solution to (\ref{eq:finiteprob}) without the $u^{\b{\tilde N}}\b a = 1$ condition. Assume $S$ and $\partial S$ satisfy the assumptions for $S^{\b j}$ in Proposition (\ref{prop:gamma}). Then, for small enough $h_\text{max}$, there are constants $A_q,B_q>0$ such that 
\begin{align*}
\norm{u^{\b {\tilde{N}}} - u}_{H^r\b S} &\le A_{q}h_{\text{max}}^{q-\frac{m_S}{2}}\norm{ u}_{d}+{B}_{q}h_{\text{max}}^{q-\frac{m_S-1}{2}-\lceil{\tilde{p}}\rceil}\norm u_{d}.
\end{align*}
\end{proposition}
\begin{proof}
Assumption (\ref{eq:regular}) applied to $u^{\b{\tilde N}} - u$ along with Proposition \ref{prop:gamma}.
\end{proof}

There are a few comments to make regarding Proposition \ref{prop:conv}. First, note that since Proposition \ref{prop:gamma} gives global estimates, not just estimates on a finite set of points as in a finite difference consistency result, the regularity or stability of the PDE (\ref{eq:fullprob}) itself is sufficient for convergence of the numerical method. We also point out that some conditions for the domain are implicitly imposed by the assumption that the solution $u$ to (\ref{eq:fullprob}) can be extended to $\H\b d$. For Lipschitz flat domains $S$, functions on $H^k\b S$ can be extended to $H^k\b {\R^m}$ \cite[see][Thm. A.4]{mclea00}. For manifolds with boundary, this theorem must be applied on patches to suitably extend $u$ to a larger manifold using a partition of unity, and then Proposition \ref{prop-cpexist} can be applied. We can construct an $H^k$ periodic function on the box $\Omega$ by multiplying the extension to $\R^m$ by suitable $C^\infty$ bump functions \cite[see][Prop. 2.25]{lee13}. 
Finally, we comment that the regularity assumption (\ref{eq:regular}) holds for elliptic problems on flat domains with unique solutions for Neumann, Dirichlet, or Robin boundary conditions, with various values of $r, \tilde p$ depending on the degree of smoothness of $\partial S$ and the coefficient functions of $\F$ \cite[see][Thms. 4.10, 4.11 and Thm. 4 of 6.3.2, respectively]{mclea00,evans10}. For the regularity of the Laplace-Beltrami problem on a closed surface, see Thm. 3.3 of \cite{dziuk13}, and for boundary value problems involving the Laplace-Beltrami operator on manifolds, see Prop 1.2, Thm 1.3, Prop 1.7, Eq. (7.6), and Prop. 7.5 in Chapter 5 of \cite{taylo23}.

\section{Numerical Experiments}\label{sec:numexp}

\subsection{Surface Poisson}

We now move on to testing convergence for a surface PDE. Specifically,
we solve a Poisson problem on a catenoid with a “wavy” edge.
The surface is given by:
\begin{align}\label{eq:para}
S=\cb{\b{\cosh\b t\cos\b s,\cosh\b t\sin\b s,t}:s\in\left[0,2\pi\right), \b{t-0.1\sin\b{3s}}\in \sb{-1,1}}.
\end{align}

We solve the Poisson problem:
\begin{align*}
-\Delta_{S}u\b{s,t} & =16\frac{\cos\b{4s}}{\cosh^{2}\b t}=:f\b{s,t}
, \ u\b{s,t}\eval_{\partial S} =\cos\b{4s}=:g\b{s,t}.
\end{align*}

The exact solution to this problem is $u\b{s,t}=\cos\b{4s}$. To solve
this, we impose Hermite-Birkhoff interpolation conditions on a point cloud
$S_{\tilde{N}}$ and boundary point cloud $\partial S_{\tilde{N}_{\partial}}$ to attempt to find a first-order Closest Point-like extension
that approximately solves the PDE (see Corollary \ref{cor:Let--and}). Note that such an extension exists due to Proposition \ref{prop-cpexist}. Our discretized problem is:
\begin{align}
\text{minimize: } & \norm{\vc b}_{2}\label{eq:lbmin}\\
\text{subject to: } & \sum_{n=1}^{N_{b}}d_{n}^{-\frac{1}{2}}b_{n}\b{\Delta\phi_{n}-\n^{*}\b{D^{2}\phi_{n}}\n}=f\text{, on \ensuremath{S_{\tilde{N}}}}\nonumber\\
 & \sum_{n=1}^{N_{b}}d_{n}^{-\frac{1}{2}}b_{n}\n^{*}\nabla\phi_{n}=0\text{, on \ensuremath{S_{\tilde{N}}}}
 ,\ \sum_{n=1}^{N_{b}}d_{n}^{-\frac{1}{2}}b_{n}\phi_{n}=g\text{, on \ensuremath{\partial S_{\tilde{N}_{\partial}}}}.\nonumber
\end{align} 

Our choice of $d_{n}$ is
$\b{\exp\b{q\sqrt{2\pi/T}}+\exp\b{q\sqrt{\norm{\vc{\omega}_{n}}_{2}}}}^{2}$
with $q=4$; this choice achieves super-algebraic convergence when there is a solution in the native space. This is since the sequences $F^{\b{j,q}}d^{-\frac{1}{2}}$ from Proposition \ref{prop:gamma} will be in $\ell^2$ for all $q$ for this problem and choice of $d$.
$T$ can be used to control the oscillation width of
the $\psi$ functions from Subsection \ref{subsec:adjrange}. We use $T=2$, which will provide fast convergence
at the cost of poorer conditioning for dense point clouds, compared to smaller values of $T$. $T$ serves the same purpose as RBF shape parameters; however, conditioning here is not as much of an issue compared to most RBF methods, meaning a wider range of parameters are suitable. We discuss this further at the end of this subsection.
The other
parameters used are $\ell=4$ for a $4\times4\times4$ extension
domain $\Omega$ and $N_{b}=27^3$ Fourier basis functions.
This choice for $N_b$ ensures that the error is primarily determined by $h_\text{max}$ rather than the truncation of the Fourier series; $N_b\gg\tilde{N}$ for this test. 

\begin{figure}[ht]
\begin{centering}
\includegraphics[scale=0.7]{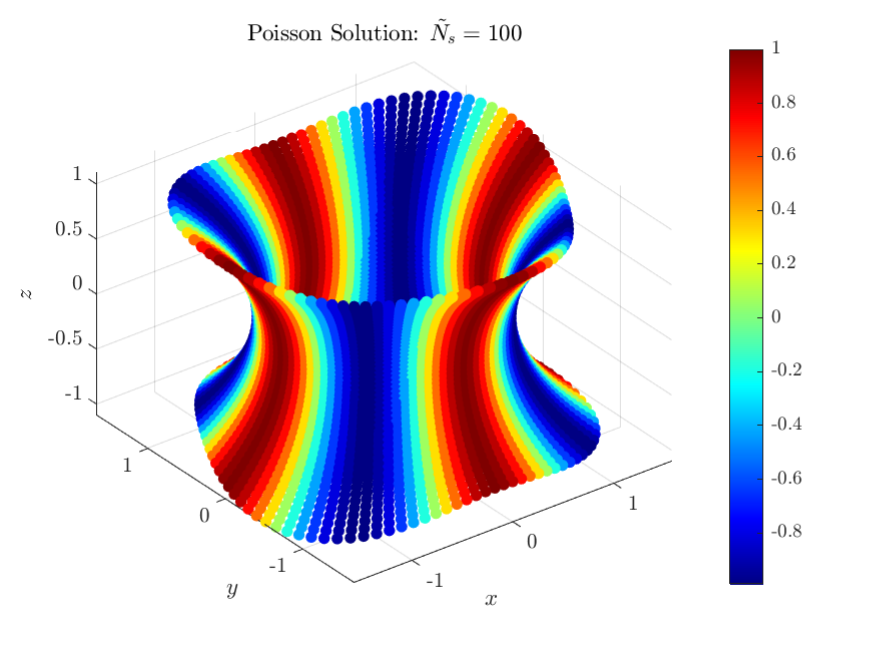}
\end{centering}
\caption{\label{fig:A-plot-of}A plot of the surface Poisson solution for $\tilde{N}_{s}=100$.}
\end{figure}

We place points on the surface by constructing a tensor grid of $\tilde{N}=\tilde{N}_s^2/2$ points on the plane of $\b{s,t}$ coordinates in $\left[0,2\pi\right)\times\sb{-1,1}$, where $\tilde{N}_{s}$ is the number of unique $s$-coordinate values so that $h_{\text{max}}\propto\tilde{N}_{s}^{-1}$. Then, we shift the values of $t$ up by $0.1\sin\b{3s}$ and map the points to $S$ with the parametrization given by (\ref{eq:para}). Importantly, the choice of points is not prescribed by the method, and we can choose our points freely as long as $h_\text{max}\to0$ as $\tilde{N}\to\infty$. For example, a randomly spaced point cloud is used in the next subsection. However, more regularly spaced point distributions tend to produce lower errors in practice and should be used when possible.  Figure
\ref{fig:A-plot-of} shows a plot of the computed solution for $\tilde{N}_{s}=100$ and Table \ref{tab:Recorded-max-error} presents a convergence test.

\begin{table}[ht]
\caption{\label{tab:Recorded-max-error}Recorded max error and estimated convergence rate on the point cloud
for the surface Poisson problem for various $\tilde{N}_{s}$ values.}
\begin{tabular*}{\columnwidth}{@{\extracolsep\fill}lll@{\extracolsep\fill}}
\toprule 
$\tilde{N}_{s}$ & Max Error & Convergence Rate\tabularnewline
\midrule

20 & $3.4619\times10^{-3}$ & N/A\tabularnewline

40 & $4.6741\times10^{-5}$ & 6.211\tabularnewline

60 & $1.7315\times10^{-6}$ & 8.128\tabularnewline

80 & $5.1884\times10^{-8}$ & 12.193\tabularnewline

100 & $3.1698\times10^{-9}$ & 12.527\tabularnewline
\bottomrule
\end{tabular*}
\end{table}

We see high-order convergence in Table \ref{tab:Recorded-max-error}. An important note is that such a low
error for $\tilde{N}_s=100$ is not possible without solving the optimization
problem directly; we use a complete orthogonal decomposition. Forming
the kernel matrix $\vc \Phi_{N_b}=\vc V_{N_{b}}\vc V_{N_{b}}^*$ results in a much more poorly conditioned linear system to solve, akin to direct RBF methods, causing convergence to stall. Note that the condition number of $\vc V_{N_{b}}\vc V_{N_{b}}^*$ is the square of the condition number of $\vc V_{N_{b}}$. The error for $\tilde{N}_s =100$ that is
obtained when solving the system with the kernel matrix naively is much larger at $4.7372\times 10^{-6}$.
Having additional, more stable options for finding a solution is a benefit
of underdetermined Fourier extensions over Hermite RBFs, while the
option to form the kernel matrix for the Fourier extensions also remains.

\subsection{Eigenvalue Problem\label{subsec-eigs}}

As mentioned in Subsection \ref{subsec:A-Note-on}, we may be able to use
the fact that the $d$-norm of interpolant solutions is bounded if
and only if a solution exists to investigate solvability numerically.
As an example, we can consider the eigenvalue problem on the unit
sphere $S=\mathbb{S}^{2}$ for the (negative) Laplace-Beltrami operator:
\[
-\Delta_{S}u-\lambda u=0\text{, for some $u\ne0$}.
\]

The eigenvalues $\lambda$ are $n\b{n+1}$ for non-negative integers $n$. To find
the eigenvalues numerically, we can solve the problem:
\begin{align}\label{eq:interpeig}
\b{-\Delta\tilde{u}-\n_{S}\cdot\b{D^{2}\tilde{u}}\n_{S}-\lambda \tilde{u}}\eval_{S_{\tilde{N}}} =0, \quad \n_{S}\cdot\nabla\tilde{u}\eval_{S_{\tilde{N}}} =0, \quad \tilde{u}\b{0,0,1} =1.
\end{align}
where $S_{\tilde N}$ is a set of $\tilde N$ randomly placed points on the sphere, and we choose the interpolant $\tilde u=\sum_{n=1}^{N_{b}}d_{n}^{-\frac{1}{2}}\tilde b_{n}\phi_{n}$ such that $\norm{\tilde {\vc b}}_2$ is minimized subject to the constraints in (\ref{eq:interpeig}).

The final condition ensures that we produce a non-zero solution. At first glance, the location of the single point $\vc a$
where we enforce $\tilde{u}\b{\vc a}=1$ may seem to be important, since
the eigenspace for a single eigenvalue may have $\tilde{u}\b{\vc a}=0$
for our selected $\vc a$. However, a random point would work with probability one, since the zero set of any eigenfunction is measure zero. That is, for arbitrary surfaces, choosing a point randomly should not cause an issue in practice. In this sphere example, symmetry ensures that there is an eigenfunction with $\tilde{u}\b{0,0,1}=1$
for any eigenvalue of $-\Delta_{S}$, and any point location would work. Another condition that forces the function to be non-zero would also be acceptable, such as setting a sum of function values or derivatives at random points to be equal to 1, but in our testing, setting $\tilde{u}\b{\vc a}=1$ for a single point $\vc a$ is sufficient. Eigenvalue multiplicities can be explored by setting multiple
points where the function must be non-zero, but we leave this for later work.

The $d$-norm of the interpolant
for various values of $\lambda$ is shown in Figure \ref{fig:Relative-error-for-1},
along with red lines at the true eigenvalues.

\begin{figure}[ht]
\begin{centering}
\includegraphics[scale=0.7]{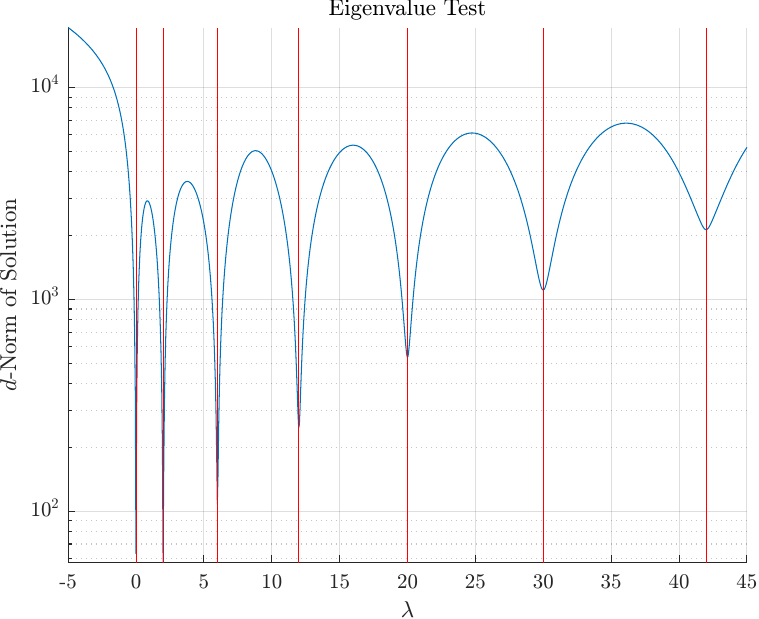}
\par\end{centering}
\caption{\label{fig:Relative-error-for-1}$d$-norm of interpolant solution
for various $\lambda$ ($\tilde{N}=400$ randomly generated points
on the sphere, $25^{3}$ Fourier basis functions on $\Omega=\protect\b{-2,2}^{3}$,
$d_{n}=\b{\exp\b{4\sqrt{2\pi/T}} + \exp\protect\b{4\sqrt{\protect\norm{\protect\vc{\omega}_{n}}_2}}}^2$, $T=4$). True eigenvalues are indicated
by red lines.}
\end{figure}

We see that the minima of the $d$-norm align with the correct eigenvalues in Figure \ref{fig:Relative-error-for-1}.
We also look specifically at the convergence of the first non-zero
eigenvalue ($\lambda=2$) in Table \ref{tab:Error-and-convergence-1};
we use a simple bisection-like search to find the local minimum. $h_\text{avg}$ is the average
distance between points on the surface and their closest points in the point cloud and is inversely proportional to $\sqrt{\tilde{N}}$.

\begin{table}[H]
\caption{\label{tab:Error-and-convergence-1}Error in the estimate of the first
non-zero eigenvalue for the sphere (randomly generated points on the
sphere, $25^{3}$ Fourier basis functions on $\Omega=\protect\b{-2,2}^{3}$,
$d_{n}=\b{\exp\b{4\sqrt{2\pi/T}} + \exp\protect\b{4\sqrt{\protect\norm{\protect\vc{\omega}_{n}}_2}}}^2$, $T=4$).}
\begin{tabular*}{\columnwidth}{@{\extracolsep\fill}llll@{\extracolsep\fill}}
\toprule 
$\tilde{N}$ & $\abs{\lambda_{\text{est}}-2}$ & Convergence Rate ($\tilde{N}$) & Convergence Rate ($h_{\text{avg}}$)\tabularnewline

\midrule 
100 & $1.7868\times10^{-2}$ & N/A & N/A\tabularnewline

150 & $3.3266\times10^{-3}$ & 4.146 & 8.292\tabularnewline

200 & $1.1514\times10^{-4}$ & 11.692 & 23.384\tabularnewline

250 & $9.2243\times10^{-6}$ & 11.313 & 22.625\tabularnewline

300 & $3.2410\times10^{-7}$ & 18.366 & 36.732\tabularnewline
\bottomrule
\end{tabular*}
\end{table}

Table \ref{tab:Error-and-convergence-1} again shows high-order convergence. Overall, this method of finding eigenvalues seems to be more reliable
and theoretically justified than attempting to produce a differentiation
matrix from interpolation followed by differentiation of the basis
functions, which can produce incorrect results, even for small $\lambda$
for a variety of collocation methods (see the discussion at the end of Subsection 2.2 from \cite{yan23}). Its use is also not limited to eigenvalues; any solvability problem relying on
parameters can be investigated similarly.

\section{Conclusions} \label{sec:conclusions}

We successfully developed, analyzed, and tested methods for Hermite-Birkhoff interpolation and PDEs on surfaces. The methods rely on underdetermined Fourier extensions and work on unstructured point clouds with very few conditions for convergence. 
We have shown that our method produces solutions that approximately solve PDEs of interest in an extremely general setup (Proposition \ref{prop:gamma}); convergence rates could then be estimated for a wide range of PDEs, including those that satisfy a regularity condition (Proposition \ref{prop:conv}). Convergence is also shown in a general setting without rates by Proposition \ref{prop:Let--be2}.

Of particular interest, our method works for surface PDEs on unstructured point clouds and is able to achieve arbitrarily high rates of convergence. For bulk problems, our method could be seen as an alteration of earlier, symmetric Hermite RBF methods \citep[see, for example,][]{sun94,frank98,liu23}. Hermite RBFs do not seem to be commonly applied to surface PDEs outside of one approach for RBF finite differences from \cite{shaw19}. The setup of traditional Hermite RBF approaches, however, can be quite difficult, which may help explain why such methods are not widespread compared to non-symmetric RBF methods. We also find that it is typically easier to work with Fourier series directly to set up the correct linear system rather than with the often unwieldy functions that arise with Hermite RBFs. 

Furthermore, our method allows for more stable methods of solution; solving the optimization problem by complete orthogonal decomposition or singular value decomposition tends to produce superior results compared to actually forming $\vc{\Phi}_{N_b}$ when the system is ill-conditioned. Again, this is related to the fact that the condition number of $\vc \Phi_{N_b} = \vc V_{N_{b}}\vc V_{N_{b}}^*$ is the square of the condition number of $\vc V_{N_b}$; forming $\vc \Phi_{N_b}$ is quite similar to using Hermite RBFs, but with positive definite functions on a box rather than all of $\R^m$.

Our future work seeks to take advantage of the flexibility of the method for problems that would be intractable using more traditional approaches. Spacetime methods have shown promise in numerical testing, as have methods for producing conformal parametrizations of surfaces defined solely by point clouds. Particularly, such as in the case of conformal mapping, the method allows for an accurate solution to be found to problems with many possible solutions. A similar approach to Subsection \ref{subsec-eigs} may also allow a range of solvability problems to be investigated numerically in a rigorously justified manner. Overall, we expect that our method could be useful for a large number of problems since its setup is nearly universal for linear PDEs (on surfaces), and since it is meshfree and high-order.

\section*{Funding}
We acknowledge the support of the Natural Sciences and
Engineering Research Council of Canada (NSERC), [funding reference numbers 579365-2023 (DV), RGPIN-2022-03302 (SR)].


\begin{thebibliography}{35}
\providecommand{\natexlab}[1]{#1}
\providecommand{\url}[1]{\texttt{#1}}
\expandafter\ifx\csname urlstyle\endcsname\relax
  \providecommand{\doi}[1]{doi: #1}\else
  \providecommand{\doi}{doi: \begingroup \urlstyle{rm}\Url}\fi

\bibitem[Adcock and Huybrechs(2019)]{adcoc19}
B.~Adcock and D.~Huybrechs.
\newblock Frames and numerical approximation.
\newblock \emph{SIAM Rev.}, 61\penalty0 (3):\penalty0 443--473, 2019.
\newblock \doi{10.1137/17M1114697}.

\bibitem[Adcock et~al.(2014)Adcock, Huybrechs, and Martin-Vaquero]{adcoc14}
B.~Adcock, D.~Huybrechs, and J.~Martin-Vaquero.
\newblock On the numerical stability of {F}ourier extensions.
\newblock \emph{Found. Comput. Math.}, 14:\penalty0 635--687, 2014.
\newblock \doi{10.1007/s10208-013-9158-8}.

\bibitem[{\'A}lvarez et~al.(2021){\'A}lvarez, Gonz{\'a}lez-Rodr{\'\i}guez, and Kindelan]{alvar21}
D.~{\'A}lvarez, P.~Gonz{\'a}lez-Rodr{\'\i}guez, and M.~Kindelan.
\newblock A local radial basis function method for the {L}aplace--{B}eltrami operator.
\newblock \emph{J. Sci. Comput.}, 86\penalty0 (3):\penalty0 28, 2021.
\newblock \doi{10.1007/s10915-020-01399-3}.

\bibitem[Boyd(2002)]{boyd03}
J.~P. Boyd.
\newblock A comparison of numerical algorithms for {F}ourier extension of the first, second, and third kinds.
\newblock \emph{J. Comput. Phys.}, 178\penalty0 (1):\penalty0 118--160, 2002.
\newblock \doi{10.1006/jcph.2002.7023}.

\bibitem[Boyd(2005)]{boyd05}
J.~P. Boyd.
\newblock {F}ourier embedded domain methods: extending a function defined on an irregular region to a rectangle so that the extension is spatially periodic and ${C}^{\infty}$.
\newblock \emph{Appl. Math. Comput.}, 160\penalty0 (2):\penalty0 591--597, 2005.
\newblock \doi{10.1016/j.amc.2003.12.068}.

\bibitem[Bruno and Paul(2022)]{bruno22}
O.~P. Bruno and J.~Paul.
\newblock Two-dimensional {F}ourier continuation and applications.
\newblock \emph{SIAM J. Sci. Comput}, 44\penalty0 (2):\penalty0 A964--A992, 2022.
\newblock \doi{10.1137/20M1373189}.

\bibitem[Bruno et~al.(2007)Bruno, Han, and Pohlman]{bruno07}
O.~P. Bruno, Y.~Han, and M.~M. Pohlman.
\newblock Accurate, high-order representation of complex three-dimensional surfaces via {F}ourier continuation analysis.
\newblock \emph{J. Comput. Phys.}, 227\penalty0 (2):\penalty0 1094--1125, 2007.
\newblock \doi{10.1016/j.jcp.2007.08.029}.

\bibitem[Chandrasekaran and Mhaskar(2015)]{chand15}
S.~Chandrasekaran and H.~Mhaskar.
\newblock A minimum {S}obolev norm technique for the numerical discretization of {PDE}s.
\newblock \emph{J. Comput. Phys.}, 299:\penalty0 649--666, 2015.
\newblock \doi{10.1016/j.jcp.2015.07.025}.

\bibitem[Chandrasekaran et~al.(2013)Chandrasekaran, Jayaraman, and Mhaskar]{chand13}
S.~Chandrasekaran, K.~Jayaraman, and H.~Mhaskar.
\newblock Minimum {S}obolev norm interpolation with trigonometric polynomials on the torus.
\newblock \emph{J. Comput. Phys.}, 249:\penalty0 96--112, 2013.
\newblock \doi{10.1016/j.jcp.2013.03.041}.

\bibitem[Chandrasekaran et~al.(2018)Chandrasekaran, Gorman, and Mhaskar]{chand18}
S.~Chandrasekaran, C.~Gorman, and H.~N. Mhaskar.
\newblock Minimum {S}obolev norm interpolation of scattered derivative data.
\newblock \emph{J. Comput. Phys.}, 365:\penalty0 149--172, 2018.
\newblock \doi{10.1016/j.jcp.2018.03.014}.

\bibitem[Chen and Ling(2020)]{chen20}
M.~Chen and L.~Ling.
\newblock Extrinsic meshless collocation methods for {PDE}s on manifolds.
\newblock \emph{SIAM J. Numer. Anal.}, 58\penalty0 (2):\penalty0 988--1007, 2020.
\newblock \doi{10.1137/17M1158641}.

\bibitem[Curtis~Jr.(1959)]{curti59}
P.~C. Curtis~Jr.
\newblock $n$-parameter families and best approximation.
\newblock \emph{Pac. J. Math.}, 9\penalty0 (4):\penalty0 1013--1027, 1959.
\newblock \doi{10.2140/pjm.1959.9.1013}.

\bibitem[Dziuk and Elliott(2013)]{dziuk13}
G.~Dziuk and C.~M. Elliott.
\newblock Finite element methods for surface {PDE}s.
\newblock \emph{Acta Numer.}, 22:\penalty0 289--396, 2013.
\newblock \doi{10.1017/S0962492913000056}.

\bibitem[Evans(2010)]{evans10}
L.~Evans.
\newblock \emph{Partial Differential Equations}, volume~19 of \emph{Graduate Studies in Mathematics}.
\newblock American Mathematical Society, Providence, USA, 2 edition, 2010.
\newblock ISBN 9781470411442.
\newblock \doi{10.1090/gsm/019}.

\bibitem[Franke and Schaback(1998)]{frank98}
C.~Franke and R.~Schaback.
\newblock Solving partial differential equations by collocation using radial basis functions.
\newblock \emph{Appl. Math. Comput.}, 93\penalty0 (1):\penalty0 73--82, 1998.
\newblock \doi{10.1016/S0096-3003(97)10104-7}.

\bibitem[Johnson(1975)]{johns75}
D.~J. Johnson.
\newblock The trigonometric {H}ermite-{B}irkhoff interpolation problem.
\newblock \emph{Trans. Amer. Math. Soc.}, 212:\penalty0 365--374, 1975.
\newblock \doi{10.2307/1998632}.

\bibitem[Kansa(1990)]{kansa90}
E.~Kansa.
\newblock Multiquadrics—a scattered data approximation scheme with applications to computational fluid-dynamics—{II} solutions to parabolic, hyperbolic and elliptic partial differential equations.
\newblock \emph{Comput. Math. Appl.}, 19\penalty0 (8):\penalty0 147--161, 1990.
\newblock \doi{10.1016/0898-1221(90)90271-K}.

\bibitem[Kreyszig(1991)]{kreys91}
E.~Kreyszig.
\newblock \emph{Introductory Functional Analysis with Applications}, volume~17.
\newblock John Wiley \& Sons, Hoboken, USA, 1991.

\bibitem[Lee(2003)]{lee13}
J.~Lee.
\newblock \emph{Introduction to Smooth Manifolds}.
\newblock Graduate Texts in Mathematics. Springer Science+Business Media, New York, USA, 2 edition, 2003.

\bibitem[Liu et~al.(2023)Liu, Li, and Hu]{liu23}
J.~Liu, X.~Li, and X.~Hu.
\newblock A novel local {H}ermite radial basis function‐based differential quadrature method for solving two‐dimensional variable‐order time fractional advection–diffusion equation with {N}eumann boundary conditions.
\newblock \emph{Numer. Methods Partial Differential Equations}, 39\penalty0 (4):\penalty0 2998--3019, 2023.
\newblock \doi{10.1002/num.22997}.

\bibitem[Mairhuber(1956)]{mairh56}
J.~C. Mairhuber.
\newblock On {H}aar's theorem concerning {C}hebychev approximation problems having unique solutions.
\newblock \emph{Proc. Am. Math. Soc.}, 7\penalty0 (4):\penalty0 609--615, 1956.

\bibitem[M\"{a}rz and Macdonald(2012)]{marz12}
T.~M\"{a}rz and C.~B. Macdonald.
\newblock Calculus on surfaces with general closest point functions.
\newblock \emph{SIAM J. Numer. Anal.}, 50\penalty0 (6):\penalty0 3303--3328, 2012.
\newblock \doi{10.1137/120865537}.

\bibitem[Matthysen and Huybrechs(2018)]{matth18}
R.~Matthysen and D.~Huybrechs.
\newblock Function approximation on arbitrary domains using {F}ourier extension frames.
\newblock \emph{SIAM J. Numer. Anal.}, 56\penalty0 (3):\penalty0 1360--1385, 2018.
\newblock \doi{10.1137/17M1134809}.

\bibitem[McLean(2000)]{mclea00}
W.~C.~H. McLean.
\newblock \emph{Strongly elliptic systems and boundary integral equations}.
\newblock Cambridge University Press, Cambridge, UK, 2000.
\newblock \doi{10.2307/3621632}.

\bibitem[Narcowich et~al.(2005)Narcowich, Ward, and Wendland]{narco05}
F.~J. Narcowich, J.~D. Ward, and H.~Wendland.
\newblock {S}obolev bounds on functions with scattered zeros, with applications to radial basis function surface fitting.
\newblock \emph{Math. Comput.}, 74\penalty0 (250):\penalty0 743--763, 2005.
\newblock \doi{10.1090/S0025-5718-04-01708-9}.

\bibitem[Petras et~al.(2018)Petras, Ling, and Ruuth]{petra18}
A.~Petras, L.~Ling, and S.~J. Ruuth.
\newblock An {RBF-FD} closest point method for solving {PDEs} on surfaces.
\newblock \emph{J. Comput. Phys.}, 370:\penalty0 43--57, 2018.
\newblock \doi{10.1016/j.jcp.2018.05.022}.

\bibitem[Piret(2012)]{piret12}
C.~Piret.
\newblock The orthogonal gradients method: A radial basis functions method for solving partial differential equations on arbitrary surfaces.
\newblock \emph{J. Comput. Phys.}, 231\penalty0 (14):\penalty0 4662--4675, 2012.
\newblock ISSN 0021-9991.
\newblock \doi{10.1016/j.jcp.2012.03.007}.

\bibitem[Ruuth and Merriman(2008)]{ruuth08}
S.~J. Ruuth and B.~Merriman.
\newblock A simple embedding method for solving partial differential equations on surfaces.
\newblock \emph{J. Comput. Phys.}, 227\penalty0 (3):\penalty0 1943--1961, 2008.
\newblock \doi{10.1016/j.jcp.2007.10.009}.

\bibitem[Shaw(2019)]{shaw19}
S.~B. Shaw.
\newblock Radial basis function finite difference approximation of the {L}aplace-{B}eltrami operator.
\newblock Master's thesis, Boise State University, 2019.
\newblock URL \url{https://scholarworks.boisestate.edu/cgi/viewcontent.cgi?article=2711&context=td}.

\bibitem[Stein et~al.(2017)Stein, Guy, and Thomases]{stein17}
D.~B. Stein, R.~D. Guy, and B.~Thomases.
\newblock Immersed boundary smooth extension ({IBSE}): A high-order method for solving incompressible flows in arbitrary smooth domains.
\newblock \emph{J. Comput. Phys.}, 335:\penalty0 155--178, 2017.
\newblock \doi{10.1016/j.jcp.2017.01.010}.

\bibitem[Sun(1994)]{sun94}
X.~Sun.
\newblock Scattered {H}ermite interpolation using radial basis functions.
\newblock \emph{Linear Algebra Appl.}, 207:\penalty0 135--146, 1994.
\newblock \doi{10.1016/0024-3795(94)90007-8}.

\bibitem[Taylor(2023)]{taylo23}
M.~E. Taylor.
\newblock \emph{Partial Differential Equations I: Basic Theory}, volume 115 of \emph{Applied Mathematical Sciences}.
\newblock Springer Nature Switzerland AG, Cham, Switzerland, 3 edition, 2023.
\newblock \doi{10.1007/978-3-031-33859-5}.

\bibitem[Wendland(2004)]{wendl04}
H.~Wendland.
\newblock \emph{Scattered Data Approximation}.
\newblock Cambridge Monographs on Applied and Computational Mathematics. Cambridge University Press, Cambridge, UK, 2004.
\newblock \doi{10.1017/CBO9780511617539}.

\bibitem[Xu and Zhao(2003)]{xu03}
J.-J. Xu and H.-K. Zhao.
\newblock An {E}ulerian formulation for solving partial differential equations along a moving interface.
\newblock \emph{J. Sci. Comput.}, 19:\penalty0 573--594, 2003.
\newblock \doi{10.1023/A:1025336916176}.

\bibitem[Yan et~al.(2023)Yan, Jiang, and Harlim]{yan23}
Q.~Yan, S.~W. Jiang, and J.~Harlim.
\newblock Spectral methods for solving elliptic {PDEs} on unknown manifolds.
\newblock \emph{J. Comput. Phys.}, 486:\penalty0 112132, 2023.
\newblock \doi{10.1016/j.jcp.2023.112132}.

\end{thebibliography}
\end{document}